\documentclass[11pt, a4paper, twoside,reqno]{amsart}
\usepackage[utf8]{inputenc}
\usepackage{mathrsfs}
\usepackage{graphics,graphicx}
\usepackage{ulem}
\usepackage{amsmath}
\usepackage{amssymb} 
\usepackage{amsthm}
\usepackage{mathabx}
\usepackage{xcolor}
\usepackage{enumerate}
\usepackage{icomma}
\usepackage{empheq}
\usepackage{pdflscape}
\usepackage{epstopdf}
\usepackage{scalerel}
\usepackage{rotating}
\usepackage[shortlabels]{enumitem}
\usepackage[bookmarks=true,breaklinks=true,bookmarksnumbered = true,colorlinks=false,hidelinks]{hyperref}

 \newtheorem{thm}{Theorem}[section]
 
 \newtheorem{lem}[thm]{Lemma} 
  
 \newtheorem{prop}[thm]{Proposition} 
 \newtheorem{cor}[thm]{Corollary} 
 \newtheorem{defn}[thm]{Definition} 
 \newtheorem{rmq}[thm]{Remark} 
 \newtheorem{ex}[thm]{Example}

 \newtheorem{claim}[thm]{Claim}

\newcommand\supp{\mathord{\mathrm{supp}}}
\newcommand\degg{\mathord{\mathrm{deg}}}
\newcommand\lex{\mathord{\mathrm{lex}}}
\newcommand\wg{\mathord{\mathrm{wg}}}
\newcommand\Occ{\mathord{\mathrm{Occ}}}

\newenvironment{amssidewaystable}
  {\begin{sidewaystable}\vspace*{.83\textwidth}\begin{minipage}{\textheight}\centering}
  {\end{minipage}\end{sidewaystable}}
  
\begin{document}
\normalem
\title{On braids and links up to link-homotopy}
	
	\author[Emmanuel Graff]{Emmanuel Graff}
	\address{Normandie Univ, UNICAEN, CNRS, LMNO, 14000 Caen, France}
	\email{emmanuel.graff@unicaen.fr}
	 	
	\date{\today}
	
	
	
	
\begin{abstract}
This paper deals with links and braids up to link-homotopy, studied from the viewpoint of Habiro's clasper calculus. More precisely, we use clasper homotopy calculus in two main directions. First, we define and compute a faithful linear representation of the homotopy braid group, by using claspers as geometric commutators. Second, we give a geometric proof of Levine's classification of 4-component links up to link-homotopy, and go further with the classification of 5-component links in the algebraically split case.
\end{abstract}
    \maketitle

\section{Introduction}

The notion of \emph{link-homotopy} was introduced in 1954 by J.W. Milnor in \cite{MilnorLinkgrp}, in the context of knot theory. It is an equivalence relation on links that allows continuous deformations during which two distinct components remain disjoint at all times, but each component can self-intersect. Any knot is link-homotopic to the trivial one, but for links with more than one component this equivalence relation turns out to be quite rich and intricate. Since Milnor’s seminal work, link-homotopy has been the subject of numerous works in knot theory see e.g. \cite{GoldsmithHomotopybraids,Levine4comp,OrrHomoInvLinks,HabeggerLinHomotopy}, but also more generally in the study of codimension 2 embeddings (and in particular \emph{knotted surfaces} in dimension 4) \cite{MasseyRolfsenHomoClassHigherdimLinks,BartelsTeichner2DimLinksNullHomo,AudouxJBWagnerCodim2EmbeddingHomo} and  \emph{link-maps} (self-immersed spheres) \cite{FennRolfsenSphereHomo4space,KirkLinkmaps4sphere,KoschorkeLinkmapsclass,SchneidermanTeichnerGroupDisjoint2sphere4space}. In this paper we are interested in the study of link-homotopy for braids and links.

\

The \emph{homotopy braid group} has been studied by many authors. In \cite{GoldsmithHomotopybraids} Goldsmith gives an example of a non-trivial braid up to isotopy that is trivial up to link-homotopy; she also gives a presentation of the homotopy braid group. A representation of the homotopy braid group is given by Humphries in \cite{HumphriesTorsion}. He uses it to show that the homotopy braid group is torsion-free for less than 6 strands. Finally the pure homotopy braid group has been studied by Habegger and Lin in \cite{HabeggerLinHomotopy} as an intermediate object for the classification of links up to link-homotopy. As further developed below, our first main result is another linear representation of the homotopy braid group (Theorem \ref{thmgammarpz}), which we prove to be faithful (Theorem \ref{thminj}) and which is computed explicitly in Theorem \ref{thmcalculgamma}.

\

We also address the problem initially posed by Milnor in \cite{MilnorLinkgrp}, of classifying links in the 3-sphere up to link-homotopy. Milnor himself answered the question for the 2 and 3-component case. Furthermore, Habegger and Lin \cite{HabeggerLinHomotopy} proposed a complete classification, using a subtle algebraic equivalence relation on pure braids, where two equivalent braids correspond to link-homotopic links. A more direct algebraic approach had been proposed by Levine \cite{Levine4comp} just before the work of Habegger--Lin in the $4$-component case. Our second main result is a new geometric proof of Levine’s classification of $4$-component links up to link-homotopy (Theorem \ref{thm4compclass}). 
This approach seems to apply, at least in principle, to links with a higher number of components:  we illustrate this in Theorem \ref{thm5compclass} with the case of \emph{algebraically split} $5$-component links (that is, $5$-component links with vanishing linking numbers).

\

The notion of \emph{clasper} was developed by Habiro in \cite{HabiroClasp}. These are surfaces in 3–manifolds with some additional structure, on which surgery operations can be performed. In \cite{HabiroClasp}, Habiro describes the clasper calculus up to isotopy, which is a set of geometric operations on claspers that yield equivalent surgery results. It is well known to experts how clasper  calculus  can  be  refined  for  the  study  of  knotted  objects  up  to  link-homotopy  (see  for  example  \cite{Flemingyasuselfck,YasuharaAkiraSelfdelteq}). This \emph{homotopy clasper calculus}, which we review in Section \ref{sectionclaspcalc}, will be the key tool for proving all the main results outlined above.

\

The rest of this paper consists of three sections.

In Section \ref{sectionclaspcalc}, we review the homotopy clasper calculus: after briefly recalling from \cite{HabiroClasp} Habiro’s clasper theory, we recall how a fundamental lemma from \cite{Flemingyasuselfck}, combined with Habiro’s work, produces a set of geometric operations on claspers having link-homotopic surgery results.

Section \ref{sectionbraid} is dedicated to the study of braids up to link-homotopy. We start by reinterpreting braids in terms of claspers. In Section \ref{sectioncombclasp} we define \emph{comb-claspers}, a family of claspers corresponding to braid commutators. They are next used to define a \emph{normal form} on homotopy braids, thus allowing us to rewrite any braid as an ordered product of comb-claspers. In Section \ref{sectionalg} after a short algebraic interlude, we give a presentation of the \emph{pure homotopy braid group} (Corollary \ref{coropresentation}), using the work of \cite{GoldsmithHomotopybraids} and \cite{MurasugiKunioKurpitaStudyofbraid} as well as the technology of claspers. Finally, we define and study in Section \ref{sectionrpz} a representation of the homotopy braid group which is in a sense the linearization of the homotopy Artin representation. We give its explicit computation in Theorem \ref{thmcalculgamma} (see also Example \ref{excalculgamma3comp} for the 3-strand case) and show its injectivity in Theorem \ref{thminj}. Moreover, from the injectivity of the representation follows the uniqueness of the normal form and thus the definition of the \emph{clasp-numbers}, a collection of braid invariant up to link-homotopy. Note that our representation has lower dimension than Humphries one. The correspondence between the two representations has not been established yet, but we wonder if our representation could open new leads on the torsion problem for more than six strands. 

The final Section \ref{sectionlinks} focuses on the study of links up to link-homotopy. The method used is based on the precise description of some operations, which generate the algebraic equivalence relation mentioned above in the classification result of Habegger and Lin \cite{HabeggerLinHomotopy}; we provide them with a topological description in terms of claspers. This new point of view allows us, for a small number of components, to describe when two braids in normal form have link-homotopic closures. We translate in terms of \emph{clasp-number variations} the action of those operations on the normal form. In this way we recover the classification results of Milnor \cite{MilnorLinkgrp} and Levine \cite{Levine4comp} for 4 or less components (Theorem \ref{thm4compclass}). Moreover, we also classify 5-component algebraically split links up to link-homotopy (Theorem \ref{thm5compclass}).

\medskip\emph{Acknowledgement}: The author thanks the referee for his/her careful reading and insightful suggestions. This work is partially supported by the project AlMaRe (ANR-
19-CE40-0001-01) of the ANR. The author thanks P. Bellingeri  and J.B. Meilhan for their great advises and helpful discussions. 

\section{Clasper calculus up to link-homotopy}\label{sectionclaspcalc}

Clasper calculus has been developed by Habiro in \cite{HabiroClasp} in the context of \emph{tangles} up to isotopy. Claspers turn out to be in fact a powerful tool to deal with link-homotopy. In this section we first define claspers and their associated vocabulary. Then we describe how to handle claspers up to link-homotopy.
\subsection{General definitions}

Let $M$ be a smooth compact and oriented $3$-manifold.
\begin{defn}\label{deftangle}
An \emph{$n$-component tangle} in $M$ is a smooth embedding of an $n$-component ordered and oriented $1$-manifold (a disjoint union of circles and intervals) 
into $M$. 
\begin{itemize}
    \item We say that two tangles are $\emph{isotopic}$ if they are related by an ambient isotopy of $M$ that fixes the boundary.
    \item We say that two tangles are \emph{$link$-$homotopic$} if there is a homotopy between them fixing the boundary, and such that the distinct components remain disjoint during the deformation.
\end{itemize}
\end{defn}
\begin{defn}\label{clasper}
A disk $T$ smoothly embedded in $M$ is called a \emph{clasper for a tangle} $\theta$ if it satisfies the following three conditions:
\begin{enumerate}[-]
\item $T$ is the embedding of a connected thickened uni-trivalent graph with a cyclic order at each trivalent vertex. Thickened univalent vertices are called \emph{leaves}, and thickened trivalent vertices, \emph{nodes}. \item $\theta$ intersects $T$ transversely, and the intersection points are in the interior of the leaves of $T$. 
\item Each leaf intersects $\theta$ in at least one point.
\end{enumerate}
\end{defn}
Diagrammatically a clasper is represented by a uni-trivalent graph corresponding to the one to be thickened. The trivalent vertices are thickened according to Figure \ref{claspdiag}. On the univalent vertices we specify how the corresponding leaves intersect $\theta$, and we also indicate how the edges are twisted using markers called $twists$ (see Figure \ref{claspdiag}).
\begin{figure}[!htbp]
    \centering
    \includegraphics[scale=0.55]{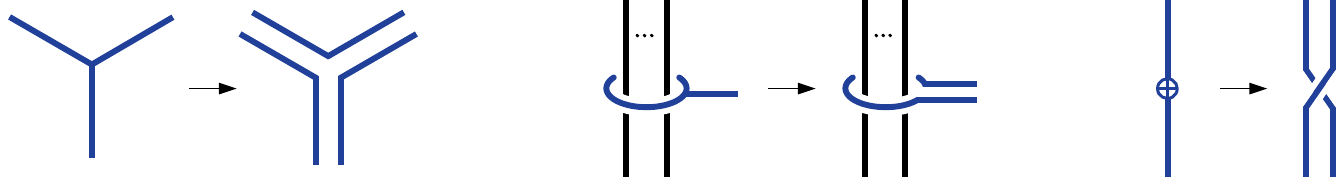}
    \caption{Local diagrammatic representation of claspers.}
    \label{claspdiag}
\end{figure}
\begin{defn}
Let $T$ be a clasper for a tangle $\theta$. We define the $\emph{degree}$ of $T$ denoted $\degg(T)$ as its number of nodes plus one, or equivalently its number of leaves minus one. The $\emph{support}$ of $T$ denoted $\supp(T)$ is defined to be the set of the components of $\theta$ that intersect $T$. 
\end{defn}
\begin{defn}
A clasper $T$ for a tangle $\theta$ is said to be \emph{simple} if every leaf of $T$ intersects $\theta$ exactly once. A leaf of a simple clasper intersecting the $l$-th component is called an \emph{$l$-leaf}.
\end{defn}
\begin{defn}
We say that a simple clasper $T$ for a tangle $\theta$ \emph{has repeats} if it intersects a component of $\theta$ in at least two points. 
\end{defn}

Given a disjoint union of claspers  $F$ for a tangle $\theta$, there is a procedure called $surgery$ detailed in \cite{HabiroClasp} to construct a new tangle denoted $\theta^F$. We illustrate on the left-hand side of Figure \ref{figsurgery} the effect of a surgery on a clasper of degree one. Now if $F$ contains some claspers with degree higher than one, we first apply the rule shown on the right-hand side of Figure \ref{figsurgery}, at each trivalent vertex: this breaks up $F$ into a disjoint union of degree one claspers, on which we can perform surgery. 
\begin{figure}[!htbp]
      \centering
      \includegraphics[scale=0.55]{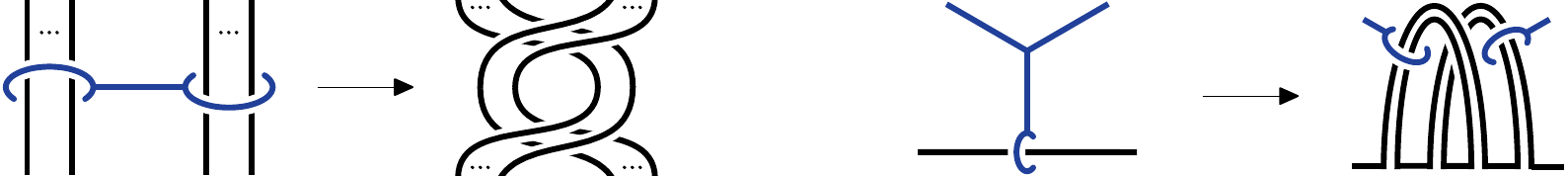}
      \caption{Rules of clasper surgery.}
      \label{figsurgery}
\end{figure}

Note that clasper surgery commutes with ambient isotopy. More precisely for $i$ an ambient isotopy and $F$ a disjoint union of claspers for a tangle $\theta$ we have that $i(\theta^{F})=(i(\theta))^{i(F)}$. This is an elementary example of $clasper\ calculus$, which refers to the set of operations on unions of a tangles with some claspers, that allow to deform one into another with isotopic surgery result. These operations are developed in \cite{HabiroClasp}, and we give in the next section the analogous calculus up to link-homotopy.

\subsection{Clasper calculus up to link-homotopy.}

In the whole section, $T$ and $S$ denote simple claspers for a given tangle $\theta$. We use the notation $T\sim S$, and say that \emph{$T$ and $S$ are link-homotopic} when the surgery results $\theta^{T}$ and $\theta^{S}$ are so. For example if $i$ is an ambient isotopy that fixes $\theta$, then $T\sim i(T)$. Moreover, if $\theta^T$ is link-homotopic to $\theta$, we say that $T$ \emph{vanishes up to link-homotopy} and we denote $T\sim \emptyset$.

We begin by recalling a fundamental lemma from \cite{Flemingyasuselfck}; more precisely, the next result is the case $k=1$ of \cite[Lemma 1.2]{Flemingyasuselfck}, where self $C_1$-equivalence corresponds to link-homotopy. 
\begin{lem}\cite[Lemma 1.2]{Flemingyasuselfck}\label{lemfeuilledouble}
If $T$ has repeats then $T$ vanishes up to link-homotopy.
\end{lem}
It is well known to the experts that combining Lemma \ref{lemfeuilledouble} with the proofs of Habiro's technical results on clasper calculus \cite{HabiroClasp}, yields the following \emph{link-homotopy clasper calculus}.\footnote{Those moves are contained in \cite{YasuharaAkiraSelfdelteq} and \cite{JBYasuMilnorHOMPFLYT} together with \cite{Flemingyasuselfck}.} 
\begin{prop}\cite[Proposition 3.23, 4.4, 4.5 and 4.6]{HabiroClasp}\label{lemclaspcalculus}
We have the following link-homotopy equivalences (illustrated in Figure \ref{clasperscalculus}).
\begin{enumerate}[(1)]
    \item If $S$ is a parallel copy of $T$ which differs from $T$ only by one twist (positive or negative), then $S\cup T\sim\emptyset$.
    \item If $T$ and $S$ have two adjacent leaves and if $T'\cup S'$ is obtained from $T\cup S$ by exchanging these leaves as depicted in (2) from Figure \ref{clasperscalculus}, then $T\cup S\sim T'\cup S'\cup \tilde T$, where $\tilde T$ is as shown in the figure.
    \item If $T'$ is obtained from $T$ by a crossing change with a strand of the tangle $\theta$ as depicted in (3) from Figure \ref{clasperscalculus}, then $T\sim T'\cup \tilde T$, where $\tilde T$ is as shown in the figure.
    \item If $T'\cup S'$ is obtained from $T\cup S$ by a crossing change between one edge of $T$ and one of $S$ as depicted in (4) from Figure \ref{clasperscalculus}, then $T\cup S\sim T'\cup S'\cup \tilde T$, where $\tilde T$ is as shown in the figure.
    \item If $T'$ is obtained from $T$ by a crossing change between two edges of $T$ then $T\sim T'$.
\end{enumerate}
\begin{figure}[!htbp]
    \centering
    \includegraphics[scale=0.70]{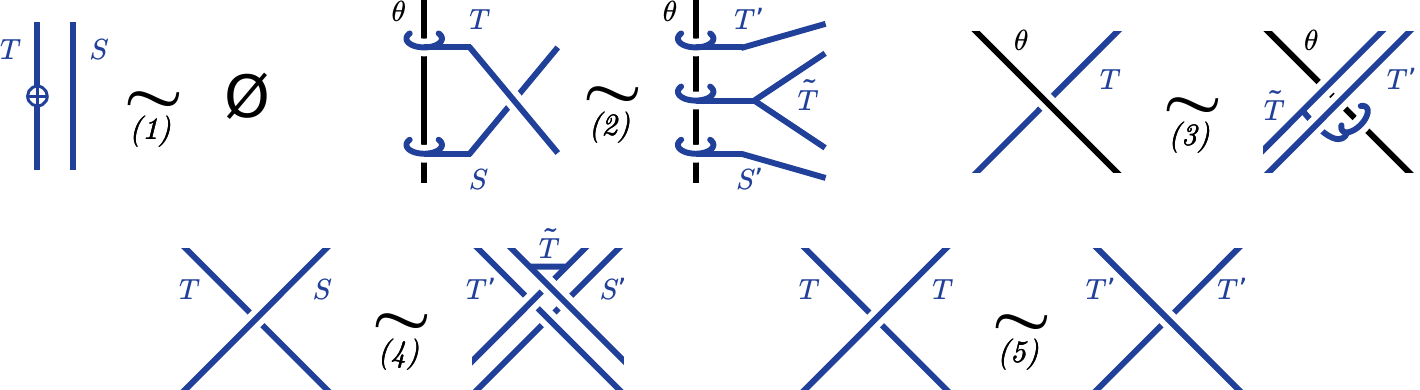}
    \caption{Basic clasper moves up to link-homotopy.}
    \label{clasperscalculus}
\end{figure}
\end{prop}
\begin{proof}[Idea of proof]
The result of \cite{HabiroClasp} used here are up to \emph{$C_k$-equivalence}, that is, up to claspers of degree up to $k$. The key observation is that, by construction, all such higher degree claspers have same support as the initial ones, hence they are claspers with repeats. Lemma \ref{lemfeuilledouble} then allows to delete them up to link-homotopy.
\end{proof}

\begin{rmq}\label{rmqclaspcalculus}
Lemma \ref{lemfeuilledouble} combined with Proposition \ref{lemclaspcalculus} gives us some further results:
\begin{enumerate}[-]
    \item First, statement $(4)$ implies that if $|\supp(T)\cap \supp(S)|\geq1$ then we can realize crossing changes between the edges of $T$ and $S$.
    \item Moreover, if $|\supp(T)\cap \supp(S)|\geq2$ thanks to statement (2) we can also exchange the leaves of $T$ and $S$.
    \item Furthermore, statement (3) allows crossing changes between $T$ and a component of $\theta$ in the support of $T$
\end{enumerate}
Indeed, in each case the clasper $\tilde T$ involved in the corresponding statement has repeats and can thus be deleted up to link-homotopy.
\end{rmq}

The next remark describes how to handle twists up to link-homotopy.
 
\begin{rmq}\label{lemtwists}
We have the following link-homotopy equivalences (illustrated in Figure \ref{lemtwistsfigure}).
\begin{enumerate}[(1)]
\setcounter{enumi}{5}
     \item If $T'$ is obtained from $T$ by moving a twist across a node then $T\sim T'$.
     \item If $T$ and $T'$ are identical outside a neighborhood of a node, and if in this neighborhood $T$ and $T'$ are as depicted in (8) from Figure \ref{lemtwistsfigure}, then $T\sim T'$.
\end{enumerate}
\begin{figure}[!htbp]\label{figtwists}
    \centering
    \includegraphics[scale=0.70]{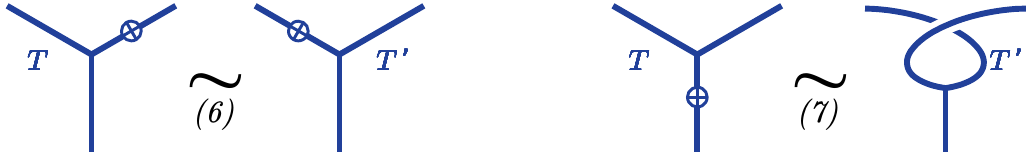}
    \caption{How to deal with twist up to link-homotopy.}
    \label{lemtwistsfigure}
\end{figure}
 \end{rmq}

\begin{rmq}\label{rmqtwist}

Remark \ref{lemtwists} allows us to bring all the twists on a same edge and then cancel them pairwise. Therefore we can consider only claspers with one or no twist.
\end{rmq}

Proposition \ref{lemclaspcalculus} together with Remark \ref{lemtwists} give us most of the necessary tools to understand clasper calculus up to link-homotopy. The missing ingredient is the relation IHX which we give in the following proposition.

\begin{prop}\label{lemIHX}\cite{HabiroClasp}
Let $T_I$, $T_H$, $T_X$ be three parallel copies of a given simple clasper that coincide everywhere outside a $3$-ball, where they are as shown in Figure \ref{IHX}. Then $T_I\cup T_H\cup T_X\sim\emptyset$. We say that $T_I$, $T_H$ and $T_X$ \emph{verify the IHX relation}.
\begin{figure}[!htbp]
    \centering
    \includegraphics[scale=0.8]{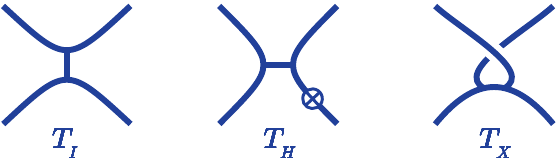}
    \caption{The IHX relation for claspers.}
    \label{IHX}
\end{figure}
\end{prop}

\section{Braids up to link-homotopy.}\label{sectionbraid}

This section is dedicated to braids up to link-homotopy. Our main result is a representation of the homotopy braid group, defined and studied using clasper calculus. In the next two subsections we introduce the main tools for this result: first the notion of comb-claspers for braids, that yields a normal form result up to link-homotopy, and next their algebraic counterpart with the family $\mathcal F$.

\subsection{Braids and comb-claspers.}\label{sectioncombclasp}

Let $D$ be the unit disk with $n$ fixed points $\{p_i\}_{i\leq n}$ on a diameter $\delta$, and $I$ the unit interval $[0,\ 1]$. Set also $I_1,\ \ldots,\ I_n$, $n$ copies of $I$, and $\underset{i\leq n}\bigsqcup I_i$ their disjoint union. 
\begin{defn}\label{braiddef}
An \emph{$n$-component braid} $\beta=(\beta_1,\ \ldots,\ \beta_n)$ is a smooth proper embedding  \[(\beta_1,\ \ldots,\ \beta_n) :\underset{i\leq n}{\bigsqcup}I_i \to D\times I\] 
such that $\beta_i(0)=(p_i,\ 0)$ and $\beta_i(1)=(p_{\pi(\beta)(i)},\ 1)$ with $\pi(\beta)$ some permutation of $\{1,\ \ldots,\ n\}$ associated to $\beta$. We also require the embedding to be monotonic, which means that $\beta_i(t)\in D\times \{t\}$ for any $t\in[0,\ 1]$. We call (the image of) $\beta_i$ the \emph{$i$-th component} of $\beta$. We say that a braid is \emph{pure} if its associated permutation is the identity.
\end{defn}
We emphasize that the braids are here oriented from top to bottom.

The set of braids up to ambient isotopy, resp. link-homotopy, equipped with the stacking operation forms a group: the \emph{braid group} denoted by $B_n$, resp. the \emph{homotopy braid group}, denoted by $\tilde B_n$. Elements of $\tilde B_n$ are called \emph{homotopy braids}. The set of pure braids up to isotopy, resp. link-homotopy, forms a subgroup of $B_n$, resp. $\tilde B_n$, denoted by $P_n$, resp. $\tilde P_n$. Note that we do not require isotopy or link-homotopy to preserve the monotonic property during the deformation.

\begin{rmq}\label{rmqstringlink}
Braids are tangles without closed components, and with boundary and monotonic conditions. 
But any (pure) tangle without closed components is link-homotopic to a (pure) braid (in the pure case, such tangles are called string links in the literature). Thus, when regarding braids up to link-homotopy we can freely consider them as tangles, i.e. we can forget the monotonic condition. 
This is useful from the clasper point of view since clasper surgery does not respect this condition in general. 
\end{rmq}
 
We introduce next \emph{comb-claspers} and their associated notation. Consider the usual representative $\mathbf 1$ of the trivial $n$-component braid given by $\mathbf1_i=\{p_i\}\times I$ for $i\in\{1,\ \ldots,\ n\}$. Denote by $(D\times I)^+$ and $(D\times I)^-$ the two half-cylinders determined by the plane $\delta\times I$, where $\delta$ is the fixed diameter on $D$. In figures, we choose $(D\times I)^+$ to be above the plane of the projection.

\begin{defn}\label{defcombclasp}
We call \emph{comb-clasper} a simple clasper without repeats for the trivial braid such that:
\begin{enumerate}[-]
    \item Every edge is in $(D\times I)^+$.
    \item The minimal path running from the smallest to the largest component of the support meets all nodes.
    \item At each node, the edge that does not belong to the minimal path leaves “to the left" as locally depicted in Figure \ref{figedgeleft}.
\end{enumerate}
\end{defn}
\begin{figure}[!htbp]
    \centering
    \includegraphics[scale=0.7]{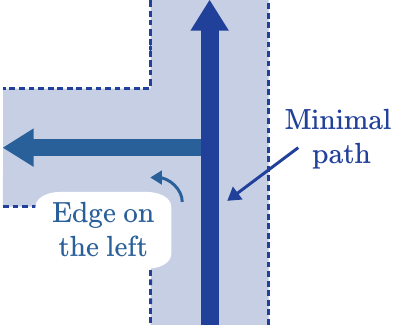}
    \caption{Local orientation at each node of a comb-clasper.}
    \label{figedgeleft}
\end{figure}
An example is given in Figure \ref{claspeigne}.

\

The second condition of Definition \ref{defcombclasp} implies that every node is related (by an edge and a leaf) to a component of $\mathbf 1$ that is not the smallest or the largest of the support. Using that, we can order the support of a comb-clasper: we start with the smallest component, then we order the components according to the order in which we meet them along the minimal path, and finally, we end with the largest one. For example in Figure \ref{claspeigne} the ordered support is $\{1,\ 2,\ 6,\ 4,\ 5,\ 8\}$. 

\

Once the ordered support $\{i_1,\ i_2,\ \ldots,\ i_l\}$ is fixed, the only remaining indeterminacy in a comb-clasper is the embedding of the edges in $(D\times I)^+$. This depends on the relative position of the edges, and on the number of twists on each of them. However, up to link-homotopy the relative position of the edges is irrelevant (by move (5) from Proposition \ref{lemclaspcalculus}). Besides, by Remark \ref{rmqtwist}, we can always suppose that a comb-clasper contains either one or no twist; moreover by Remark \ref{lemtwists} we can freely assume that the potential twist is located on the  edge connected to the $i_l$-th component. We can thus unambiguously (up to link-homotopy) denote by $(i_1,i_2,\cdots,i_l)$ the comb-clasper with such a twist and by $(i_1,i_2,\cdots,i_l)^{-1}$ the untwisted one; we call them respectively \emph{twisted} and \emph{untwisted} comb-claspers. For example the twisted comb-clasper $(126458)$ is illustrated in Figure \ref{claspeigne}.

\begin{figure}[!htbp]
    \centering
    \includegraphics[scale=0.8]{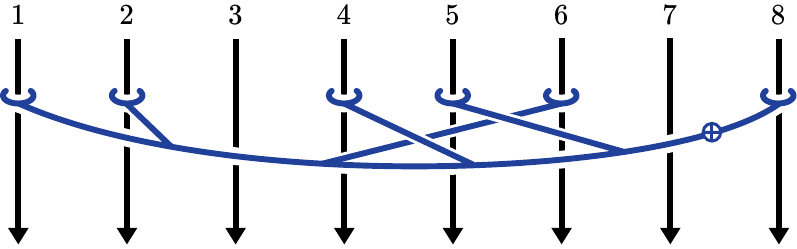}
    \caption{The twisted comb-clasper $(126458)$.}
    \label{claspeigne}
\end{figure}

In what follows we blur the distinction between comb-claspers and the result of their surgery up to link-homotopy. From this point of view a comb-clasper is a pure homotopy braid and the product $(\alpha)(\alpha')$ of two comb-claspers is the product $\mathbf1^{(\alpha)}\mathbf1^{(\alpha')}$. In particular according to move (1) from Proposition \ref{lemclaspcalculus} the inverse of a comb-clasper $(\alpha)$ is given by $(\alpha)^{-1}$.

\begin{defn}
We say that a pure homotopy braid $\beta\in\tilde P_n$ given by a product of comb-claspers\\ $\beta=(\alpha_1)^{\pm1}(\alpha_2)^{\pm1}\cdots(\alpha_m)^{\pm1}$ is :
    \begin{itemize}
    \item \emph{stacked} if $(\alpha_i)=(\alpha_j)$ for some $i\leq j$ implies that $(\alpha_i)=(\alpha_k)$ for any $i\leq k\leq j$,
    \item \emph{reduced} if it contains no redundant pair i.e. two consecutive factors are not the inverse of each other.
    \end{itemize}
    If $\beta$ is reduced and stacked we can then rewrite $\beta=\prod(\alpha_i)^{\nu_i}$ for some integers $\nu_i$ and with $(\alpha_i)\neq(\alpha_j)$ for any $i\neq j$. Moreover, given an order on the set of twisted comb-claspers, we say that a reduced and stacked writing is a \emph{normal form of $\beta$ for this order} if $(\alpha_i)\leq(\alpha_j)$ for all $i\leq j$.
\end{defn}
We stress that the notion of normal form is relative to a given order on the set of twisted comb-claspers. The following example will be relevant for Section \ref{sectionlinks}.
\begin{ex}\label{exnormformclasp}
Given two twisted comb-claspers $(\alpha)=(i_1\cdots i_l)$ and $(\alpha')=(i'_1\cdots i'_{l'})$ we can choose the order  $(\alpha)\leq(\alpha')$ defined by:
\begin{itemize}
    \item $\max(\supp(\alpha))<\max(\supp(\alpha')),$ or
    \item $\max(\supp(\alpha))=\max(\supp(\alpha'))$ and $\degg(\alpha)<\degg(\alpha'),$ or
    \item $\max(\supp(\alpha))=\max(\supp(\alpha'))$ and $\degg(\alpha)=\degg(\alpha')$ and $i_1\ldots i_l<_{\lex}i'_1\ldots i'_l,$
\end{itemize} 
where $<_{\lex}$ denotes the lexicographic order.
With respect to this order the normal form of an element $\beta\in\tilde P_4$ is given by 12 integers $\{\nu_{12},\ \ldots,\ \nu_{1324}\}$ as follows:
$$\beta=(12)^{\nu_{12}}(13)^{\nu_{13}}(23)^{\nu_{13}}(123)^{\nu_{123}}(14)^{\nu_{14}}(24)^{\nu_{24}}(34)^{\nu_{34}} (124)^{\nu_{124}}(134)^{\nu_{134}}(234)^{\nu_{234}}
(1234)^{\nu_{1234}}(1324)^{\nu_{1324}}.$$
\end{ex}

The next result is shown by the same arguments as for Theorem 4.3 of \cite{YasuharaAkiraSelfdelteq}.\footnote{Although a different notion of comb-clasper is used in \cite{YasuharaAkiraSelfdelteq}, the strategy of proof is strictly the same.}
\begin{thm}\label{normalformex}\cite[Theorem 4.3]{YasuharaAkiraSelfdelteq}
Any pure homotopy braid $\beta\in\tilde P_n$ can be expressed in a normal form, for any order on the set of twisted comb-claspers.
\end{thm}

\subsection{Algebraic counterpart.}\label{sectionalg}
\subsubsection{Reduced group and commutators.}

For any $a,b$ in a group we denote $[a,b]:=aba^{-1}b^{-1}$.
\begin{defn}\label{defredgrp}
Let $G$ be a group generated by $\{x_1,\ \ldots,\ x_n\}$. We define $J_G\triangleleft G$ to be the normal subgroup generated by elements of the form $[x_i,\lambda x_i \lambda^{-1}]$, for all $i\in\{1,\ \ldots,\ n\}$, and for all $\lambda\in G$. We call \emph{reduced quotient}, the quotient $G/J_G$ and we denote it by $\mathcal RG$. This definition depends on the choice of the generators $\{x_1,\ \ldots,\ x_n\}$. 
\end{defn}
In what follows we work essentially with the free group $F_n$ on $n$ generators $x_1,\ \ldots,\ x_n$. The reduced quotient $\mathcal RF_n=F_n/J$ of the free group is called \emph{reduced free group}, where $J:=J_{F_n}$.
\begin{defn}
A \emph{commutator in $x_1 ,\ \ldots,\ x_n$ of weight $k$} $(k\geq1)$ is an element of $F_n$ defined recursively, as follows:
\begin{itemize}
    \item The commutators of weight one are $x_1 ,\ \ldots,\ x_n$.
    \item The commutators of weight $k$ are words $[C _1,C _2 ]$ where $C_1,\ C_2$ are commutators verifying $k=\wg(C_1 )+\wg(C_2)$ where $\wg(C$) denotes the weight of $C$.
\end{itemize}
\end{defn}
\begin{defn}
We denote $\Occ_i (C) = r$ and we say that $x_i$ \emph{occurs $r$ times} in a commutator $C$ if one of the following holds:
\begin{itemize}
    \item If $C = x_j$, then $r = 1$ if $i = j$ and $r = 0$ if $i\neq j$.
    \item If $C = [C_1,C_2 ]$, then $r = \Occ_i (C_1 ) + \Occ_i (C_2 )$.
\end{itemize}
We say that a commutator $C$ \emph{has repeats} if $\Occ_i(C)>1$ for some $i$.
We call \emph{support} of the commutator $C$, the set of indices $i$ such that $\Occ_i(C)>0$ and we denote it $\supp(C)$.
\end{defn}

The following is a reformulation of Definition \ref{defredgrp} that is used throughout the paper.
\begin{prop}\label{propcommurep}
\cite[Proposition 3]{Levine4comp} The subgroup $J$ is generated by commutators in $x_1,\ \ldots,\ x_n$ with repeats. Hence these commutators are trivial in the reduced free group.
\end{prop}
The notion of \emph{basic commutators} was first introduced in \cite{Phall} and was further studied in \cite{LyndonRogerSchuppPaulCobigrpthry,Mhall,MagnusKarrassSolitar} to describe the lower central series of the free group. It was then naturally adapted in \cite{Levine4comp} to the framework of the reduced free group. In the next definition we set a well chosen family of commutators. This family will replace \emph{reduced basic commutators} from \cite{Levine4comp} and will follow us throughout the whole paper.
\begin{defn}\label{basiscommudef}
Let us define the following family of commutators without repeats in $\mathcal{R}F_n$: $$\mathcal F=\left\{[i_1,\ \ldots,\ i_l]\ \middle|\  i_1<i_k,\  2\leq k\leq l\right\}_{l\leq n}.$$
Here, we use the notation $[i_1,i_2,\cdots,\ i_l]:=[[\cdots[[x_{i_1},x_{i_2}],x_{i_3}],\cdots,x_{i_{l-1}}],x_{i_l}]$. This is a finite set and we can thus choose an arbitrary order on it, $\mathcal{F}=\{[\alpha_1],\ [\alpha_2],\ \ldots,\ [\alpha_m]\}$.
\end{defn}

\begin{ex}\label{exnormalformcommu}
For two commutators $[\alpha]=[i_1\cdots i_l]$ and $[\alpha']=[i'_1\cdots i'_{l'}]$ we can choose the order given by $[\alpha]\leq[\alpha']$ if:
\begin{itemize}
    \item $\wg(\alpha)<\wg(\alpha'),$ or
    \item $\wg(\alpha)=\wg(\alpha')$ and $i_1\ldots i_l<_{\lex}i'_1\ldots i'_l.$
\end{itemize}
With respect to this order the normal form of an element $\omega\in\mathcal RF_3=\langle x_1,\ x_2,\ x_3\rangle$ is given by 8 integers $\{e_{1},\ \ldots,\ e_{8}\}$ as follows:
\[\omega=[1]^{e_1}[2]^{e_2}[3]^{e_3}[12]^{e_4}[13]^{e_5}[23]^{e_6}[123]^{e_7}[132]^{e_8}.\] 
\end{ex}
The following theorem is a kind of reduced analogue of Hall's basis theorem \cite[Theorem 11.2.4]{Mhall}. It is to be compared with \cite[Proposition 6]{Levine4comp}, where a different family of commutators is used, see Remark \ref{rmqpasinj}. 
\begin{thm}\label{basiscommuthm}
For any word $\omega\in\mathcal RF_n$ there exists a unique ordered set of integers $\{e_1,\ \ldots,\ e_m\}$ associated to the ordered family of commutators $\mathcal{F}=\{[\alpha_1],\ [\alpha_2],\ \ldots,\  [\alpha_m]\}$ such that $$\omega=[\alpha_1]^{e_1}[\alpha_2]^{e_2}\cdots [\alpha_m]^{e_m}.$$
\end{thm}
\begin{proof}
We first show for $\omega\in\mathcal RF_n$ the existence of a decomposition $\omega=\prod_{\alpha\in\mathcal{F}}[\alpha]^{e_\alpha}$. We recall that two commutators commute up to commutators of strictly higher weight, and any commutator of weight strictly bigger than $n$ has repeats and is then trivial according to Proposition \ref{propcommurep}. Thus it is sufficient to express any commutator $C$ as a product of commutators in $\mathcal{F}$. To do so we use the three following relations in $\mathcal RF_n$.
\begin{enumerate}[(i)]
    \item $[X,Y]^{-1}=[Y,X]=[X^{-1},Y]=[X,Y^{-1}]$ with $X,Y$ commutators.
    \item $[X,[Y,Z]]=[[X,Y],Z]\cdot[[X,Z],Y]^{-1}$ with $X,Y,Z$ commutators.
    \item $[UV,X]=[U,X][V,X]$ with $U,V$ commutators such that $\supp(U)\cap \supp(V)\neq\emptyset$.
\end{enumerate}
Relation (i) allows us to move the generator $x_{i_1}$ with $i_1=\min(\supp(C))$ at the desired position; we obtain $C=[\cdots[x_{i_1},C_1],\cdots ,C_k]^{\pm1}$. Relations (i) and (ii) are used to decrease the weight of the commutator $C_i$ in this expression. We start with $C_1=[C_1',C_2']$ supposing its weight is bigger than one, and we get:
\begin{align*}
    C&=[\cdots[x_{i_1},[C_1',C_2']],\cdots,C_k]^{\pm1}\\
    &=[\cdots[[x_{i_1},C_1'],C_2']\cdot[[x_{i_1},C_2'],C_1']^{-1},\cdots,C_k]^{\pm1}\\
    &=[\cdots[[x_{i_1},C_1'],C_2'],\cdots,C_k]^{\pm1}[\cdots[[x_{i_1},C_2'],C_1']^{-1},\cdots,C_k]^{\pm1}\\
    &=[\cdots[[x_{i_1},C_1'],C_2'],\cdots,C_k]^{\pm1}[\cdots[[x_{i_1},C_2'],C_1'],\cdots,C_k]^{\mp1}.
\end{align*} 
Since $\wg(C_1')<\wg(C)$ and $\wg(C_2')<\wg(C)$ we know that by iterating this operation on the new terms we can rewrite $C$ as a product of commutators of the form $[\cdots[x_{i_1},x_{i_2}],C_2],\cdots ,C_k]$. We finish by repeating the process on $C_2,\ \ldots,\ C_k$.

\

To prove the unicity of the decomposition we work with the unit group $U_n$ of the ring of power series in non-commuting variables $X_1,\ \ldots,\ X_n$. More precisely we consider its quotient $\Tilde U_n$ in which the monomials $X^{\alpha}=X_{\alpha_1}X_{\alpha_2}\cdots X_{\alpha_n}$ vanish when they have repetition (i.e. $\alpha_i=\alpha_j$ for some $i\neq j$). The elements in $\tilde U_n$ are of the form $1+Q$ with $Q$ a sum of monomials of degree higher than one, and $(1+Q)^{-1}=1+\bar Q$ with $\bar Q=-Q+Q^2-Q^3+\cdots(-1)^nQ^n$. 
Now we can define the \emph{reduced Magnus expansion} $\tilde{M}$. This is a homomorphism from the reduced free group $\mathcal RF_n$ to $\tilde U_n$, defined by $\tilde{M}(x_i)=1+X_i$. The following computation shows that $\tilde M$ respects the relations of the reduced free group, meaning that $\tilde{M}([x_i,\lambda x_i \lambda^{-1}])=1$ for any generator $x_i$ and any $\lambda$ in $\mathcal{R}F_n$:
\begin{align*}
    \tilde M(\lambda x_i \lambda^{-1})\tilde M(x_i) &=\Big(\tilde M(\lambda)(1+X_i)\tilde M(\lambda^{-1})\Big)(1+X_i)\\
    &=1+X_i+\tilde M(\lambda)X_i\tilde M(\lambda^{-1})\\
    &=(1+X_i)\Big(\tilde M(\lambda)(1+X_i)\tilde M(\lambda^{-1})\Big)\\
    &=\tilde M(x_i)\tilde M(\lambda x_i \lambda^{-1}).
\end{align*}

An easy induction on the weight $l$ of $[\alpha]\in\mathcal F$ gives the following:
\begin{claim}\label{factcool}
For every $[\alpha]=[\alpha_1,\cdots,\alpha_l]\in\mathcal F$, $\tilde M([\alpha])=1+X^{\alpha}+Q_l(X_{\alpha_1},\cdots,X_{\alpha_l})$ where $Q_{l}$ is a sum of monomials of degree $l=\mbox{wg}([\alpha])$ not starting by $X_{\alpha_1}$, and where each variable $X_{\alpha_i}$ for $i\in\{1,\ \ldots,\ l\}$ appears exactly once. 
\end{claim}
\noindent

Now, we take $\omega=\prod_{\alpha\in\mathcal{F}}[\alpha]^{e_\alpha}=\prod_{\alpha\in\mathcal{F}}[\alpha]^{e'_\alpha}$ two decompositions of an element  $\omega\in\mathcal{R}F_n$. We prove by induction on the weight of $[\alpha]$ that $e_{\alpha}=e_\alpha'$ for any commutator $[\alpha]\in\mathcal F$. Suppose that $e_{\alpha}=e'_{\alpha}$ for any $[\alpha]$ of weight $< k$ and compare the coefficients of monomial $X^\alpha$ in both $\tilde M(\prod_{\alpha\in\mathcal{F}}[\alpha]^{e_\alpha})$ and $\tilde M(\prod_{\alpha\in\mathcal{F}}[\alpha]^{e'_\alpha})$ for a fixed commutator $[\alpha]$ of degree $k$. According to Claim \ref{factcool}, commutators of weight $>k$ do not contribute to this coefficient and the only contributing weight $k$ commutator is $[\alpha]$ itself with coefficient $e_\alpha$, resp. $e'_\alpha$. Commutators of weight $<k$ may also contribute to this coefficient but the induction hypothesis ensures that the contribution is the same in both expressions. This proves that $e_{\alpha}=e'_{\alpha}$ for any $[\alpha]$ of weight $k$ and concludes the proof.
\end{proof}
\begin{rmq}\label{rmqpasinj}
Unlike Levine's proof of \cite[Proposition 6]{Levine4comp}, this proof does not require Hall’s basis theorem \cite[Theorem 11.2.4]{Mhall}. 
\end{rmq}


\begin{defn}\label{deflinearization}
To the ordered set of commutators $\mathcal{F}=\{ [\alpha_1],\ \ldots,\ [\alpha_m]\}$ in $\mathcal RF_n$  we associate a $\mathbf Z$-module $\mathcal V$ formally generated by $\{\alpha_1,\ \ldots,\ \alpha_m\}$. We also define the linearization map $\phi:\mathcal RF_n\rightarrow\mathcal V$ by:
\[\phi(\omega)=e_1\alpha_1+\cdots+e_m\alpha_m\quad\quad\quad\mbox{ where }[\alpha_1]^{e_1}\cdots[\alpha_m]^{e_m}\mbox{ is the normal form of }\omega. \]
We keep calling “commutators" the generators of $\mathcal V$ and we define the support and the weight of $\alpha$ to be those of $[\alpha]$. 
\end{defn}
We stress that the normal form and the linearization map $\phi$ both depend on the ordering on $\mathcal F$. 
\begin{lem}\label{lemrank}
The $\mathbf Z$-module $\mathcal V$ is of rank, $$rk(\mathcal{V})=\sum_{0\leq l\leq k<n}\frac{k!}{l!}.$$
Moreover we can decompose $\mathcal V$ into a direct sum of submodules $\mathcal V_i$ generated by the commutators of weight $i$. Then we obtain that: \[rk(\mathcal{V}_i)=\sum_{i-1\leq k< n}\frac{k!}{(k-i+1)!}.\]
\end{lem}
\begin{proof}
The first equality comes by counting the cardinality of $\mathcal F$. To do so we first count the elements $[\alpha]$ with first term $\alpha_1=k$. To choose $\alpha_2,\ \alpha_3,\ \ldots,\  \alpha_l$ with $0\leq l<n-k$ we only have to respect the condition that $\alpha_1<\alpha_i$. Thus they can be freely chosen in $\{k+1,\ \ldots,\ n\}$ and therefore: 
$$rk(\mathcal{V})=\sum_{k=1}^n\sum_{l=0}^{n-k}\frac{(n-k)!}{(n-k-l)!}=\sum_{k=0}^{n-1}\sum_{l=0}^{k}\frac{k!}{(k-l)!}=\sum_{k=0}^{n-1}\sum_{l=0}^{k}\frac{k!}{l!}.$$
For the second equality, we follow the same kind of reasoning, but this time $\alpha_1=k$ must be chosen in $\{1,\ \ldots,\ n-i+1\}$, then we choose the $i-1$ last numbers $\alpha_2,\ \ldots,\ \alpha_i$ without restriction in $\{k+1,\ \ldots,\ n\}$. We obtain:
$$rk(\mathcal{V}_i)
=\sum_{k=1}^{n-i+1}\frac{(n-k)!}{(n-k-i+1)!}=\sum_{k=i-1}^{n-1}\frac{k!}{(k-i+1)!}.$$
\end{proof}
\subsubsection{Braid groups.} 

In this section we use the usual Artin braid generators $\sigma_i$ for $i\in\{1,\ \ldots,\ n-1\}$ illustrated in Figure \ref{artingene} and the usual pure braid generators $A_{ij}=\sigma_{j-1}\sigma_{j-2}\cdots\sigma_{i+1}\sigma_i^2\sigma_{i+1}^{-1}\cdots\sigma_{j-2}^{-1}\sigma_{j-1}^{-1}$ for $1\leq i<j\leq n$ illustrated in Figure \ref{puregene}.
\begin{figure}[!htbp]
    \begin{minipage}[c]{.46\linewidth}
        \centering
    \includegraphics[scale=0.7]{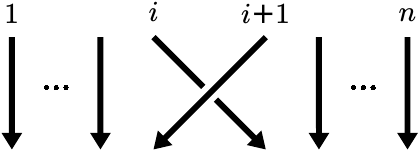}
        \caption{The Artin generator $\sigma_i$.}
        \label{artingene}
    \end{minipage}
    \hfill%
    \begin{minipage}[c]{.46\linewidth}
        \centering
    \includegraphics[scale=0.7]{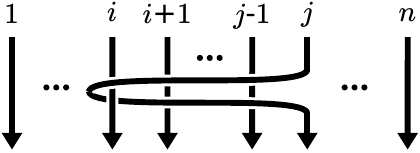}
        \caption{The pure braid generator $A_{ij}$.}
        \label{puregene}
    \end{minipage}
\end{figure}

The following theorem is based on the result of \cite{GoldsmithHomotopybraids}.
\begin{thm}\label{thmBnhomo}
Let $J\triangleleft B_n$ denote the normal subgroup generated by all elements of the form $[A_{ij},\lambda A_{ij}\lambda^{-1}]$ where $\lambda$ belongs to $P_n$. We obtain the homotopy braid group $\tilde B_n$ as the quotient:
$$\tilde B_n=B_n/J.$$
\end{thm}
\begin{proof}
In \cite{GoldsmithHomotopybraids}, the homotopy braid group $\tilde B_n$ appears as the quotient $B_n/J'$, where $J'\triangleleft B_n$ is the normal subgroup generated by elements of the form $[A_{ij},\lambda A_{ij}\lambda^{-1}]$ where $\lambda$ belongs to the normal subgroup generated by $\{A_{1,j},\ \ldots,\ A_{j-1,j}\}$. Our result relies on the observation that $J=J'$. Obviously $J'\subset J$ thus we only need to show that $J\subset J'$. This is equivalent to showing that for any $\Lambda\in P_n$, $A_{ij}$ and $\Lambda A_{ij} \Lambda^{-1}$ commute up to link-homotopy. Let us remind that $A_{ij}$ is the surgery result $\mathbf1^{(ij)}$ of the comb-clasper $(ij)$. Thus the conjugate $\Lambda A_{ij} \Lambda^{-1}$ is the surgery result of the clasper $C=\iota(ij)$, where $\iota$ is the ambient isotopy sending $\Lambda\Lambda^{-1}$ to the trivial braid $\mathbf1$. Now it is clear that $\supp(C)=\supp(ij)$, hence according to Remark \ref{rmqclaspcalculus}, $(ij)C\sim C(ij)$ and the result is proved.
\end{proof}

In order to obtain a similar result for the pure homotopy braid group we need the following.
\begin{lem}\label{lemJnormalygene}
The subgroup $J\triangleleft B_n$ normally generated in $B_n$ by elements of the form $[A_{ij},\lambda A_{ij} \lambda^{-1}]$ for $\lambda\in P_n$, seen as a subgroup of $P_n$, coincides with the normal subgroup of $P_n$ generated by elements of the form  $[A_{ij},\lambda A_{ij} \lambda^{-1}]$ for $\lambda\in P_n$.
\end{lem}
\begin{proof}
For $k\in\{1,\ \ldots,\ n-1\}$, $1\leq i<j\leq n$ and $\lambda \in P_n$ we compute:
$$\sigma_k[A_{ij},\lambda A_{ij}\lambda^{-1}]\sigma_k^{-1}=\left\{
\begin{array}{lcl}
[A_{i+1j},\lambda_1 A_{i+1j}\lambda_1^{-1}]&&\text{if }i=k\text{ and }j\neq k+1 \\

[A_{i+1j},\lambda_2 A_{i+1j}\lambda_2^{-1}]&&\text{if }j=k\\

A_{kk+1}[A_{i-1j},\lambda_3 A_{i-1j}\lambda_3^{-1}]A_{kk+1}^{-1}&&\text{if }i=k+1\\

A_{kk+1}[A_{ij-1},\lambda_4 A_{ij-1}\lambda_4^{-1}]A_{kk+1}^{-1}&&\text{if }i\neq k\text{ and }j=k+1\\

[A_{ij},\lambda A_{ij}\lambda^{-1}]&&\text{otherwise},
\end{array}\right.$$
with $\lambda_i\in P_n$ for $i\in\{1,\ 2,\ 3,\ 4\}$. Therefore the conjugates $\sigma_k[A_{ij},\lambda A_{ij}\lambda^{-1}]\sigma_k^{-1}$ are always conjugates of $[A_{i'j'},\lambda' A_{i'j'}(\lambda')^{-1}]$ in $P_n$ for some $1\leq i'<j'\leq n$ and $\lambda'\in P_n$ and the proof is done.
\end{proof}
\begin{cor}\label{coropresentation}
Let $J\triangleleft P_n$ be the normal subgroup generated by elements of the form $[A_{ij},\lambda A_{ij} \lambda^{-1}]$ for any $\lambda\in P_n$. We obtain the pure homotopy braid group $\tilde P_n$ as the following quotient:
$$\tilde P_n=P_n/J=\mathcal RP_n.$$
This induces the following presentation for $\tilde P_n$:
  $$\tilde P_n=\scaleleftright[1.3ex]{\langle}{A_{ij}\left|
  \begin{array}{lll}
    [A_{rs},A_{ij}]=1 & & r<s<i<j\mbox{ or } r<i<j<s\\  
    \left[A_{rs},A_{rj}\right]=[A_{rj},A_{sj}]=[A_{sj},A_{rs}]& & r<s<j\\
    \left[A_{ri},A_{sj}\right]=[[A_{ij},A_{rj}],A_{sj}]& &r<s<i<j\\  
    \left[A_{ij},\lambda A_{ij} \lambda^{-1}\right]=1 & & i<j \mbox{ and } \lambda\in \tilde P_n
 \end{array}\right.}{\rangle}.$$
\end{cor}
\begin{proof}
The quotient statement is a direct consequence of Proposition \ref{thmBnhomo} and Lemma \ref{lemJnormalygene}. The presentation is obtained from that of \cite[Theorem 3.8]{MurasugiKunioKurpitaStudyofbraid} re-expressed in terms of commutator and using the relation $[A_{rs},A_{ij}^{-1}]=[A_{rs},A_{ij}]^{-1}$ which holds in $\tilde P_n$.
\end{proof}

We next recall two classical representations of braid groups that are known to be faithful (see \cite{ArtinBraid} and \cite{HabeggerLinHomotopy} for more details).
\begin{defn}
We call \emph{Artin representation} the homomorphism $\rho :B_n\rightarrow Aut(F_n)$ defined as follows:\[ \rho(\sigma_i): \left \{ \begin{array}{llll}
    x_i&\mapsto&x_{i+1},&\\
    x_{i+1} &\mapsto&x_{i+1}x_ix_{i+1}^{-1},&\\
    x_k&\mapsto&x_k&\mbox{if }k\notin\{i,\ i+1\}.
 \end{array}
 \right.\]
Similarly the homomorphism $\tilde \rho :\tilde{B}_n\rightarrow Aut(\mathcal RF_n)$ defined by the same expressions is called the \emph{homotopic Artin representation}.
 \end{defn}
\subsection{A linear faithful representation of the homotopy braid group.}\label{sectionrpz}

\subsubsection{Algebraic definition}
Let $GL(\mathcal{V})$ be the general linear group of the $\mathbf Z$-module $\mathcal{V}$ introduced in Definition \ref{deflinearization}.
In order to define the linear representation $\gamma:\tilde{B}_n\rightarrow GL(\mathcal{V})$, we state the following preparatory lemma. Let us denote by $N_j$ the subgroup normally generated by $x_j$ in $\mathcal{R}F_n$ for $j\in\{1,\ \ldots,\ n\}$; note that $N_j$ is an abelian group.
\begin{lem}\label{lemgammarpz}
Let $\beta\in\tilde B_n$ be a homotopy braid. For any commutator $C\in N_j$, if the product $[\alpha_1]^{e_1}\cdots [\alpha_m]^{e_m}$ is a normal form of $\tilde{\rho}(\beta)(C)$ then we have that $e_i=0$ if $[\alpha_i]\notin N_{\pi^{-1}(\beta)(j)}$. Here $\pi^{-1}(\beta)(j)$ is the image of $j$ under the permutation induced by $\beta^{-1}$.
\end{lem}
In other words in the image by $\tilde{\rho}(\beta)$ of $C\in N_j$, $x_{\pi^{-1}(\beta)(j)}$ occurs in each factor of the normal form.
\begin{proof}
The proof comes from the fact that any element of $N_j$ is sent by $\tilde\rho(\beta)$ to an element of $N_{\pi^{-1}(\beta)(j)}$. This is clear for the Artin generators $\sigma_i$ and so it is for any braid $\beta$. Thus we conclude using the fact that the normal form $\omega=C_1^{e_1}\cdots C_m^{e_m}$ of any element $\omega\in N_k$, for any $k$ contains only commutators in $N_k$. To see this we use the homomorphism of $\mathcal RF_n$ defined by $x_k \mapsto 1$ which sends the normal form of $\omega$ to $\mathbf1$.
 \end{proof}

Recall from Definition \ref{deflinearization} the linearization map $\phi:\mathcal RF_n\to\mathcal V$.
\begin{thm}\label{thmgammarpz}
The map $$\gamma:\tilde{B}_n\rightarrow GL(\mathcal{V})$$ defined for $\beta\in\tilde B_n$ and $[\alpha]\in\mathcal F$ by $\gamma(\beta)(\alpha)=\phi\circ\tilde\rho(\beta)([\alpha])$
is a well defined homomorphism. Moreover $\gamma$ does not depend on the chosen order on $\mathcal F$.
\end{thm}
\begin{proof}
Since $\phi$ is not a homomorphism in general, it is not clear that $\gamma$ is a representation. Yet we do have that $\gamma(\beta\beta')=\gamma(\beta)\gamma(\beta')$ for any two homotopy braids $\beta$ and $\beta'$.
Let $[\alpha]$ be a commutator in $\mathcal{F}$ and $\alpha$ its corresponding commutator in $\mathcal V$. We choose some $j\in \supp([\alpha])$ so that $[\alpha]$ is in $N_j$. Set $\gamma(\beta')(\alpha)=\sum_ie_i\alpha_i$ for some commutators $\alpha_i\in\mathcal{V}$ associated to the commutators $[\alpha_i]\in\mathcal{F}$ and some integers $e_i$. Then we have that
$$\gamma(\beta\beta')(\alpha)=\phi\circ\tilde\rho(\beta)\tilde\rho(\beta')([\alpha])=\phi\circ\tilde\rho(\beta)\Big(\prod_i[\alpha_i]^{e_i}\Big)=\phi\Big(\prod_i\tilde\rho(\beta)([\alpha_i])^{e_i}\Big).$$ 
Now, using Lemma \ref{lemgammarpz} we know that $[\alpha_i]$ is in $N_{\pi^{-1}(\beta')(j)}$ for any $i$. Besides, Lemma \ref{lemgammarpz} implies that any commutator in the normal form of $\tilde\rho(\beta)([\alpha_i])$ is in the abelian group $N_{\pi^{-1}(\beta\beta')(j)}$ for any $i$. But note that for $C_1,\ \ldots,\ C_k$ a collection of commutators in $\mathcal{F}$ such that $[C_i,C_j]=1$ for any $i,\ j$ we have that $\phi(C_1\cdots C_k)=\phi(C_1)+\cdots+\phi(C_k)$. Hence $\phi$ behaves like a homomorphism on the product $\prod_i\tilde\rho(\beta)([\alpha_i])^{e_i}$, and finally, $$\phi\Big(\prod_i\tilde\rho(\beta)([\alpha_i])^{e_i}\Big)=\sum_ie_i\phi\Big(\tilde\rho(\beta)([\alpha_i])\Big)=\sum_ie_i\gamma(\beta)(\alpha_i)=\gamma(\beta)\Big(\sum_i e_i(\alpha_i)\Big)=\gamma(\beta)\gamma(\beta')(\alpha).$$ This shows that $\gamma$ is a well defined homomorphism.

To prove the independence  on the chosen order on $\mathcal F$ we use Lemma \ref{lemgammarpz} again. For any $\beta\in\tilde B_n$ and any $[\alpha]\in\mathcal F$, all the commutators in the normal form of $\tilde\rho(\beta)([\alpha])$ commute with each other. In particular if we set two orderings $\{[\alpha_1],\ \ldots,\ [\alpha_m]\}$ and $\{[\alpha_{\sigma(1)}],\ \ldots,\ [\alpha_{\sigma(m)}]\}$ on $\mathcal F$ then the two associated normal forms $$\tilde\rho(\beta)([\alpha])=[\alpha_1]^{e_1}\cdots [\alpha_m]^{e_m}=[\alpha_{\sigma(1)}]^{e'_{\sigma(1)}}\cdots [\alpha_{\sigma(m)}]^{e'_{\sigma(m)}}$$ satisfy $e_i=e'_i$ for any $i$ and therefore $\phi\circ\tilde\rho=\phi'\circ\tilde\rho$ for the two linearization maps $\phi$ and $\phi'$ associated to the orderings.
\end{proof}
\begin{rmq}
The homomorphism $\gamma$ is in fact injective. Since $\phi$ is clearly injective, this can be shown using the injectivity of $\tilde\rho$, proved in \cite{HabeggerLinHomotopy}. However we will give below another proof of this result in Theorem \ref{thminj} using clasper calculus, which in turn reproves the injectivity of $\tilde{\rho}$. Furthermore our approach by clasper calculus allows for explicit computations of the representation, as shown in the next section. 
\end{rmq}
\subsubsection{Clasper interpretation}

We first give a topological interpretation of the Artin, resp. homotopy Artin, representation. We can see the free group $F_n$, resp. reduced free group $\mathcal RF_n$, on which $B_n$, resp. $\tilde B_n$, acts as the fundamental group, resp. the reduced fundamental group, of the complement of the $n$-component trivial braid. Therefore an element of $F_n$, resp. $\mathcal RF_n$, can also be seen as the homotopy, resp. the \emph{reduced homotopy}\footnote{Here by \emph{reduced homotopy class}, we mean the image in the reduced quotient of the homotopy class of an element.}, class of an $(n+1)$-th component in this complement. On the diagram, we place this new strand to the right of the braid and we label it by “$\infty$". Thus, the generators $x_i$ of $F_n$ (resp $\mathcal RF_n$) are given by the pure braids $A_{i\infty}$ shown in Figure \ref{claspgénérateur}, which can be reinterpreted with the comb-claspers $(i,\infty)$ depicted in the same figure. There and in subsequent figures, we simply represent with a circled “$\infty$" the leaf intersecting the $\infty$-th component.
\begin{figure}[!htbp]
    \centering
    \includegraphics[scale=0.7]{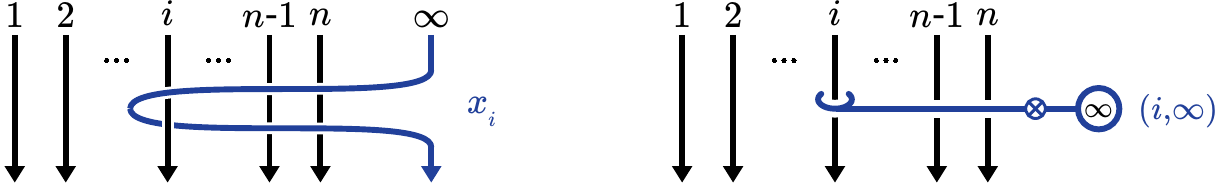}
    \caption{Pure braid and clasper interpretations of the generator $x_i$.}
    \label{claspgénérateur}
\end{figure}

In this context the image $\rho(\beta)$ of an element $\beta\in B_n$, resp. $\tilde B_n$, is given on a generator $x_i\in F_n$, resp. $\mathcal RF_n$, by considering the conjugation $\beta\mathbf1^{(i,\infty)} \beta^{-1}$ illustrated in Figure \ref{Artinrpz}. 
\begin{figure}[!htbp]
    \centering
    \includegraphics[scale=0.7]{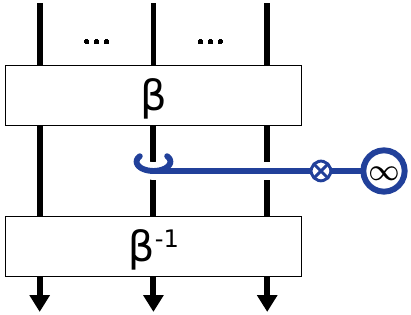}
    \caption{Clasper interpretation of the Artin representation.}
    \label{Artinrpz}
\end{figure}
Then we apply an isotopy, transforming $\beta\mathbf1\beta^{-1}$ into $\mathbf1$. By doing so the clasper $(i,\infty)$ is deformed into a new clasper which we are able to reinterpret as an element of $F_n$ or $\mathcal RF_n$. More precisely in the link-homotopic case we have a nice correspondence between the family $\mathcal{F}$ and the comb-claspers with $\infty$ in their support, by the following proposition.
\begin{prop}\label{propclapscomu}
Let $(\alpha)=(i_1\cdots i_{n-1}\infty)$ and $(\alpha')=(i_1\cdots i_{n-1}i_n\infty)$ be two comb-claspers. Then we have the relation: $$(\alpha')\sim[(\alpha),(i_n\infty)]=(\alpha)\cdot (i_n\infty)\cdot (\alpha)^{-1}\cdot (i_n\infty)^{-1}.$$
\end{prop}
For example in Figure \ref{exclaspcommutateur} we illustrate the equivalence $(1254\infty)\sim[(125\infty),(4\infty)]$.
\begin{figure}[!htbp]
     \centering
     \includegraphics[width=0.9\linewidth]{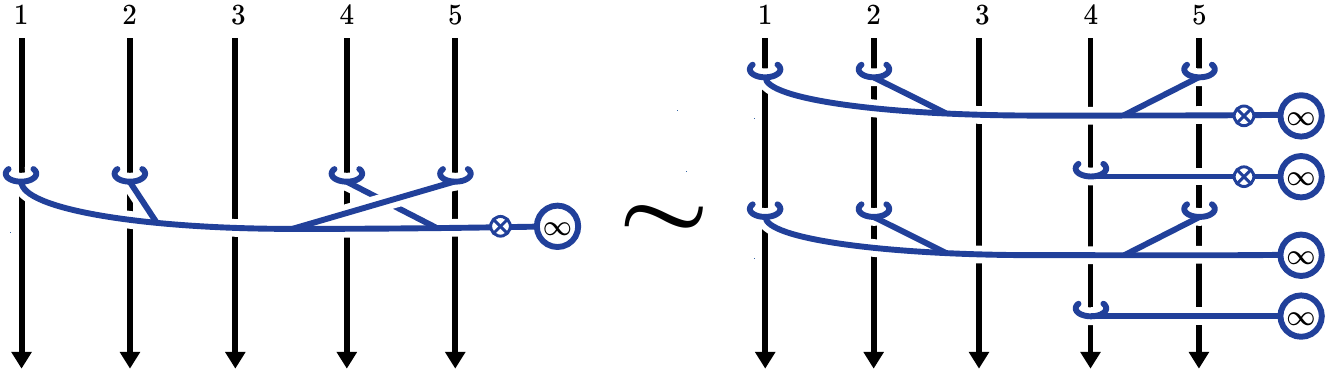}
     \caption{The comb-clasper $(1254\infty)$ is link-homotopic to the commutator $[(125\infty),(4\infty)]$.}
     \label{exclaspcommutateur}
\end{figure}
\begin{proof}
Consider the product of comb-claspers $\alpha\cdot (i_n\infty)\cdot \alpha^{-1}\cdot (i_n\infty)^{-1}$ (as for example on the right-hand side of Figure \ref{exclaspcommutateur}). First we use move $(2)$ from Proposition \ref{lemclaspcalculus} to exchange the $\infty$-th leaves of $(i_n\infty)$ and $(\alpha)^{-1}$; this move creates an extra comb-clasper, which is exactly $(\alpha')$. Now by Remark \ref{rmqclaspcalculus} we can freely move $(\alpha')$  and finish exchanging the edges of $(\alpha_n\infty)$ and $(\alpha)^{-1}$, thus obtaining the product $(\alpha)\cdot(\alpha)^{-1}\cdot(\alpha')\cdot(i_n\infty)\cdot(i_n\infty)^{-1}\sim(\alpha')$.
\end{proof}

By iterating this proposition we obtain a correspondence between the commutators $[\alpha]\in \mathcal F$ (or $\alpha\in\mathcal V$) and the comb-claspers $(\alpha,\infty)$. For example the equivalence $(1254\infty)\sim[[[(1\infty),(2\infty)],(5\infty)],(4\infty)]$ corresponds to $[1254]=[[[x_1,x_2],x_5],x_4]$ in $\mathcal RF_n$.

In this way, we obtain an explicit procedure to compute our representation $\gamma$ using clasper calculus, as follows. As illustrated in the proof of Theorem \ref{thmcalculgamma} below, the computation of $\gamma(\beta)(\alpha)$ with $\gamma$ the representation, $\beta\in\tilde B_n$ and $\alpha\in\mathcal V$, goes in 3 steps: 
\begin{description}
    \item[Step 1] Consider the conjugate of the comb-clasper $(\alpha,\infty)$ by the braid $\beta$.
    \item[Step 2] Use clasper calculus to re-express this conjugate as an ordered union of comb-claspers with $\infty$ in their support (the order comes from the order on $\mathcal F$).
    \item[Step 3] The number of parallel copies of a given comb-clasper in this product is the coefficient of the associated commutator in $\gamma(\beta)(\alpha)$. 
\end{description} 
We apply in Theorem \ref{thmcalculgamma} this procedure\footnote{A program that computes explicitly the representation $\gamma$ is available on \cite{siteweb}.} for each generator $\sigma_i\in\tilde B_n$ and each commutator in $\mathcal V$. The image of commutator $(i_1,i_2,\cdots,i_l):=\phi([i_1,i_2,\cdots,i_l])\in\mathcal V$ by the map $\gamma(\sigma_i)$ depends on the position of the indices $i$ and $i+1$ in the sequence $i_1,\ i_2,\ \ldots,\ i_l$. 

\begin{thm}\label{thmcalculgamma}
For suitable sequences $I$, $J,\ K$ in $\{1,\ \ldots,\ n\}\backslash\{i,\ i+1\}$, $I\neq\emptyset$, we have:
  \[ \gamma(\sigma_i): \left \{ \begin{array}{llll}
    (I)   &\mapsto&  (I) &(a) \\
    
    (J,i,K)   &\mapsto&  (J,i+1,K) &(b)\\

    (i+1,K)   &\mapsto&  (i,K) + (i,i+1,K)&(c)\\
    (I,i+1,K)   &\mapsto&  (I,i,K) + (I,i,i+1,K) - (I,i+1,i,K) &(d)\\
    
    (I,i,J,i+1,K)   &\mapsto& (I,i+1,J,i,K)&(e)\\
    (I,i+1,J,i,K)   &\mapsto& (I,i,J,i+1,K)&(f)\\
    (i,J,i+1,K)   &\mapsto& \sum_{J'\subseteq J}(-1)^{|J'|+1}(i,\overline{J'},i+1,J\backslash J',K)&(g)\\
 \end{array}
 \right.
 \]
where in (g), the sum is over all (possibly empty) subsequences $J'$ of $J$, and $\overline{J'}$ denotes the sequence obtained from $J'$ by reversing the order of its elements, see Example \ref{exsumseq}.
\end{thm}

\begin{ex}\label{exsumseq}
If $J=(J_1,\ J_2,\ J_3)$ and $K=\emptyset$ in (g), then 
$\gamma(\sigma_i)$ maps $(i,J,i+1)$ to :
$$-(i,i+1,J_1,J_2,J_3)+(i,J_1,i+1,J_2,J_3)+(i,J_2,i+1,J_1,J_3)+(i,J_3,i+1,J_1,J_2)$$$$-(i,J_2,J_1,i+1,J_3)-(i,J_3,J_1,i+1,J_2)-(i,J_3,J_2,i+1,J_1)+(i,J_3,J_2,J_1,i+1).$$
The proof below explains how this follows from the IHX relations of Figure \ref{calculgamma(g)'}.
\end{ex}
\begin{proof}[Proof of Theorem \ref{thmcalculgamma}]
Following the procedure given above, we consider the conjugate $\sigma_i^{-1}(\alpha,\infty)\sigma_i$ and apply clasper calculus to turn it into a union of comb-claspers. 

For (a) it is clear that $(I,\infty)$ commutes with $\sigma_i$, passing over or next to it. The computation of (b)  is given by a simple isotopy of the braid shown in Figure \ref{calculgamma(b)}.
  \begin{figure}[!htbp]
    \centering
    \includegraphics[scale=0.46]{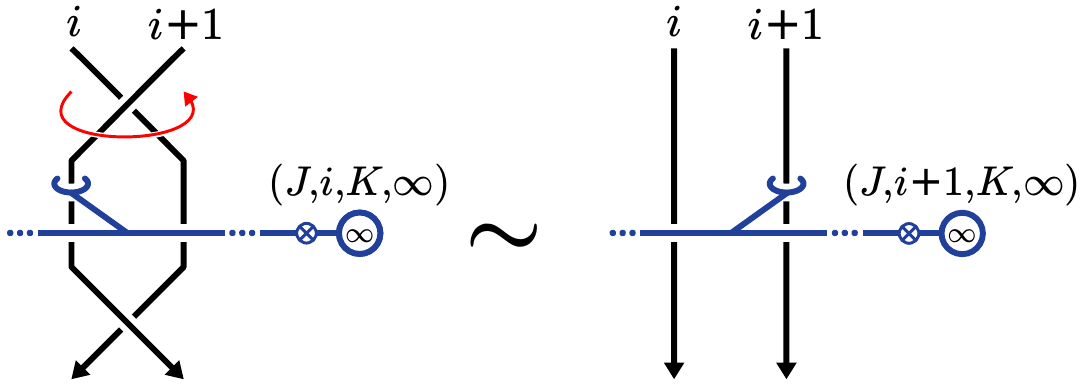}
    \caption{Computation of (b).}
    \label{calculgamma(b)}
\end{figure}

The proofs of (c) and (d) are similar and are given in Figures \ref{calculgamma(c)} and \ref{calculgamma(d)} respectively. There, the first equivalence is an isotopy, and the second one is given by move $(2)$ from Proposition \ref{lemclaspcalculus}. For (d) there is a further step given by an $IHX$ relation.
 \begin{figure}[!htbp]
    \centering
    \includegraphics[scale=0.46]{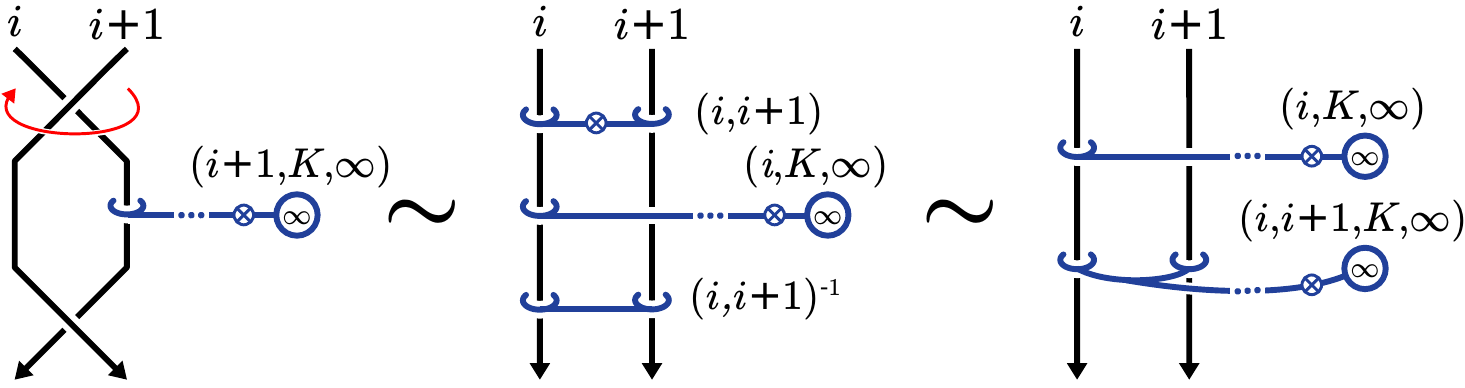}
    \caption{Computation of (c).}
    \label{calculgamma(c)}
\end{figure}
 \begin{figure}[!htbp]
    \centering
    \includegraphics[scale=0.46]{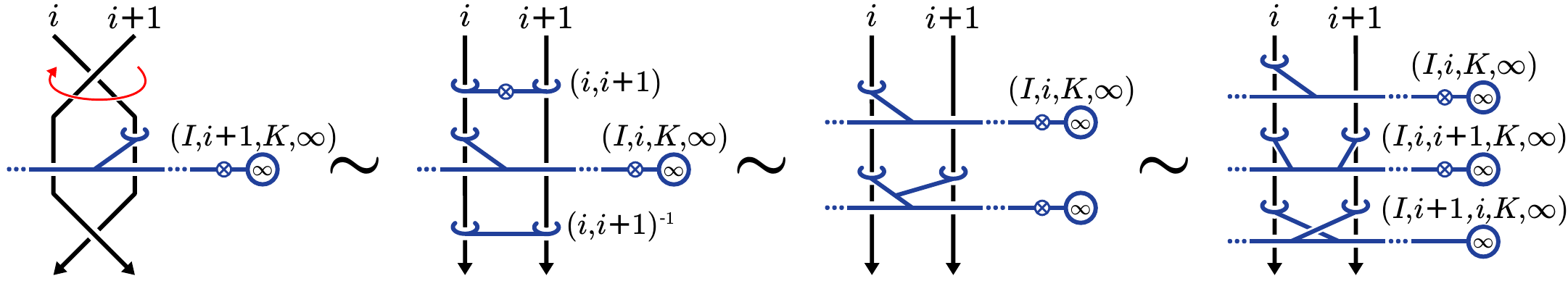}
    \caption{Computation of (d).}
    \label{calculgamma(d)}
\end{figure}

For (e) and (f) we apply the same isotopy as Figure \ref{calculgamma(b)} on components $i$ and $i+1$, thus interchanging  $(I,i,J,i+1,K)$ and $(I,i+1,J,i,K)$. Note that we also need a crossing change between the $(i+1)$-th component and a clasper edge, which is possible according to Remark \ref{rmqclaspcalculus}.
 
Proving (g) is the last and hardest part and goes in two steps. The first step is illustrated in Figure \ref{calculgamma(g)}: we proceed as before with an isotopy and a crossing change, then we use move (8) of Remark \ref{lemtwists}. This turns $\sigma_i (i,J,i+1,K,\infty) \sigma_i^{-1}$ into a new clasper which is not a comb-clasper.
 \begin{figure}[!htbp]
    \centering
    \includegraphics[scale=0.44]{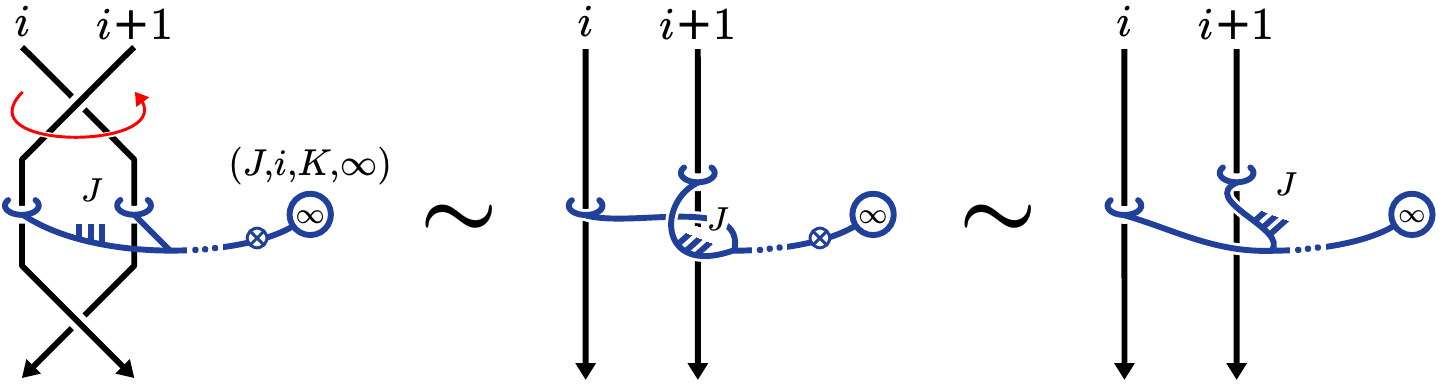}
    \caption{Turning $\sigma_i(i,J,i+1,K,\infty)\sigma_i^{-1}$ into a new clasper.}
    \label{calculgamma(g)}
\end{figure}

In the second step, we use the IHX relations repeatedly to turn this new clasper into a product of comb-claspers. This is illustrated in  Figure \ref{calculgamma(g)'} where $J=(J_1,J_2,J_3)$. We conclude by simplifying the twists with Remark \ref{rmqtwist}.
 \begin{figure}[!htbp]
    \centering
    \includegraphics[width=\linewidth  ]{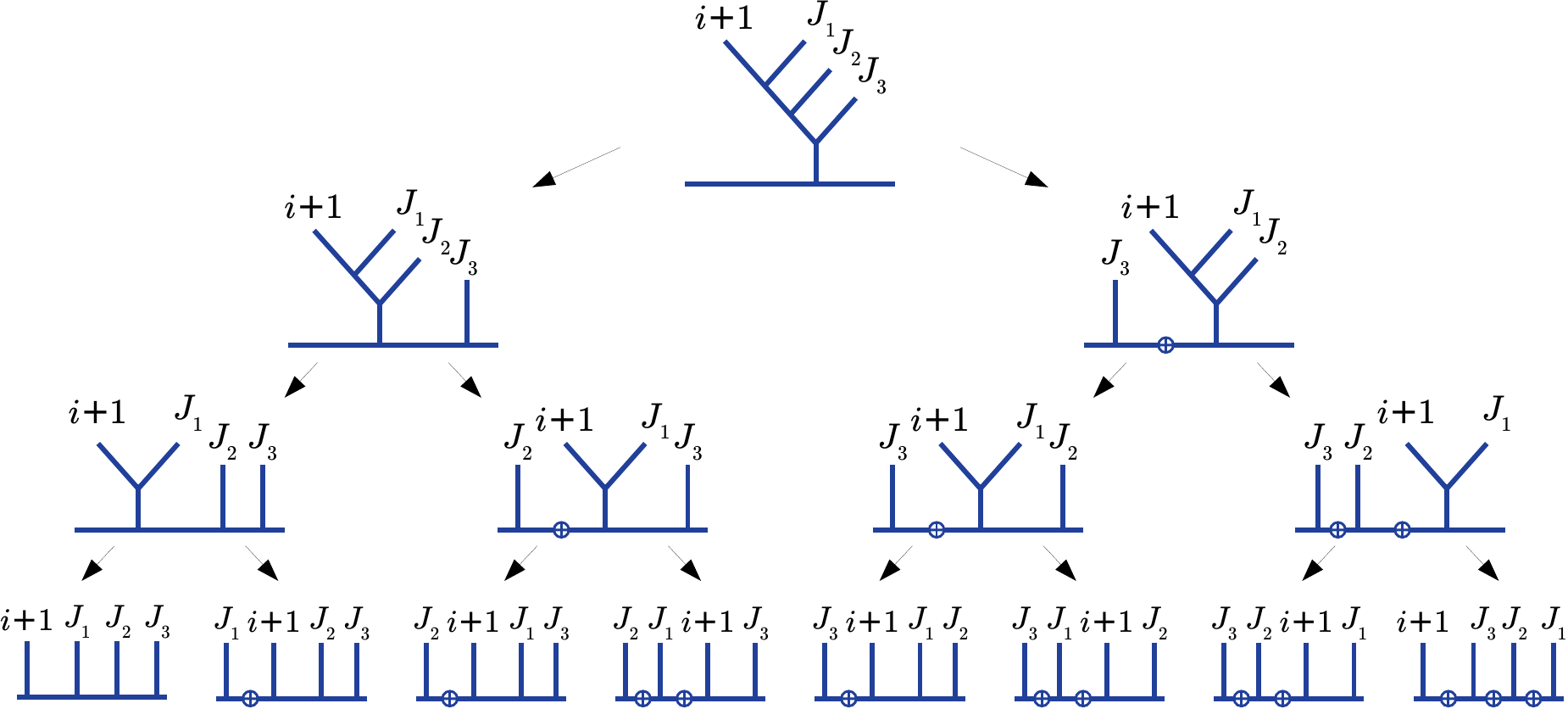}
    \caption{Iterated IHX relations.}
    \label{calculgamma(g)'}
\end{figure}
 \end{proof}
\begin{ex}\label{excalculgamma3comp}
We illustrate Theorem \ref{thmcalculgamma} on the 3-component homotopy braid group $\tilde{B}_3$. To do so, we set $(1),\ (2),\ (3),\ (12),\ (13),\ (23),\ (123),\ (132)$ to be the generators of $\mathcal V$, with the order of Example \ref{exnormalformcommu}, and we compute $\gamma$ on the Artin generators $\sigma_1,\ \sigma_2$:
 \[ \begin{array}{ll}
 \gamma(\sigma_1)(1)=(2),  &\gamma(\sigma_2)(1)=(1), \\
 \gamma(\sigma_1)(2)=(1)+(12),  &\gamma(\sigma_2)(2)=(3),\\
 \gamma(\sigma_1)(3)=(3), & \gamma(\sigma_2)(3)=(2)+(23),\\
 \gamma(\sigma_1)(12)=-(12), & \gamma(\sigma_2)(12)=(13),\\
 \gamma(\sigma_1)(13)=(23), &\gamma(\sigma_2)(13)=(12)+(123)-(132),\\
 \gamma(\sigma_1)(23)=(13)+(123),&\gamma(\sigma_2)(23)=-(23),\\
 \gamma(\sigma_1)(123)=-(123),&\gamma(\sigma_2)(123)=(132),\\
 \gamma(\sigma_1)(132)=-(123)+(132), &\gamma(\sigma_2)(132)=(123).\\
 \end{array} \]
That gives us the following matrices:
\[\begin{array}{cc}
\gamma(\sigma_1)=\begin{pmatrix}
0 & 1 & 0 & 0 & 0 & 0 & 0 & 0  \\
1 & 0 & 0 & 0 & 0 & 0 & 0 & 0  \\
0 & 0 & 1 & 0 & 0 & 0 & 0 & 0  \\
0 & 1 & 0 & -1 & 0 & 0 & 0 & 0  \\
0 & 0 & 0 & 0 & 0 & 1 & 0 & 0 \\
0 & 0 & 0 & 0 & 1 & 0 & 0 & 0\\
0 & 0 & 0 & 0 & 0 & 1 & -1 & -1 \\
0 & 0 & 0 & 0 & 0 & 0 & 0 & 1 
\end{pmatrix}, 
&\gamma(\sigma_2)=\begin{pmatrix}
1 & 0 & 0 & 0 & 0 & 0 & 0 & 0  \\
0 & 0 & 1 & 0 & 0 & 0 & 0 & 0  \\
0 & 1 & 0 & 0 & 0 & 0 & 0 & 0  \\
0 & 0 & 0 & 0 & 1 & 0 & 0 & 0  \\
0 & 0 & 0 & 1 & 0 & 0 & 0 & 0 \\
0 & 0 & 1 & 0 & 0 & -1 & 0 & 0\\
0 & 0 & 0 & 0 & 1 & 0 & 0 & 1 \\
0 & 0 & 0 & 0 & -1 & 0 & 1 & 0 
\end{pmatrix}.
\end{array}
\]
\end{ex}

The global shape of these matrices was predicted by Theorem \ref{thmcalculgamma}. Indeed in general we have the following.

\begin{prop}\label{propgammamatrix}
For $\beta\in\tilde{B}_n$ a homotopy braid, the matrix associated to $\gamma(\beta)$ in the basis $\mathcal{F}$, endowed with the order of Example \ref{exnormalformcommu}, is given by a lower triangular block matrix of the following form:
\[\begin{pmatrix}
B_{1,1} & 0 & \cdots & 0\\
B_{2,1} & B_{2,2} & \cdots & 0\\
\vdots&  \vdots & \ddots & \vdots\\
B_{n,1} & B_{n,2} & \cdots & B_{n,n}
\end{pmatrix}\]
where $B_{i,i}$ is a finite order matrix of size $rk(\mathcal{V}_i)=\sum_{i-1}^{n-1}\frac{k!}{(k-i+1)!}$ which is the identity when $\beta$ is pure. Moreover $B_{1,1}$ corresponds to the left action by permutation $k \mapsto\pi^{-1}(\beta)(k)$, and $B_{2,2}$ corresponds to the left action on the set $\{(k,\ j)\}_{k<j}$ given by:
\[(k,\ j)\mapsto\left\{\begin{array}{cc}
    \big(\pi^{-1}(\beta)(k),\ \pi^{-1}(\beta)(j)\big)& \mbox{if } \pi^{-1}(\beta)(k)<\pi^{-1}(\beta)(j),\\
    -\big(\pi^{-1}(\beta)(j),\ \pi^{-1}(\beta)(k)\big)& \mbox{if }\pi^{-1}(\beta)(j)<\pi^{-1}(\beta)(k). \end{array}\right.\]
\end{prop}
\begin{proof}
The triangular shape is a direct consequence of Theorem \ref{thmcalculgamma}. Indeed, the chosen order respects the weight, and Theorem \ref{thmcalculgamma} shows that $\gamma$ maps a commutator of weight $k$ to a sum of commutators of weight at least $k$. Proposition \ref{lemrank} gives the size of the square diagonal blocks $B_{i,i}$. The fact that these diagonal blocks are the identity when $\beta$ is a pure braid may need some more explanations. We only need to show this result on the generators $\beta=A_{i,j}=\mathbf1^{(i,j)}$. By Proposition \ref{lemclaspcalculus}, conjugating $(\alpha,\infty)$ by $(i,j)$ may only create a clasper $(\alpha',\infty)$ of strictly higher degree. This shows that $\gamma(\beta)(\alpha)=(\alpha)+\mbox{(strictly higher weight commutators)}$ so that $B_{i,i}$ is the identity. The block matrix $B_{1,1}$ describes the action on degree one comb-claspers modulo claspers of higher degree: the claim follows on an easy verification on the generators $\sigma_i$. Similarly the claim on the block matrix $B_{2,2}$ amounts to focusing on degree two comb-claspers.
\end{proof}
In order to prove the injectivity of $\gamma$ we need the following preparatory lemma. 
\begin{lem}\label{lemsondage}
Let $(i_1,\cdots,i_l)$ be a comb-clasper. We have $$\gamma\big(\mathbf{1}^{(i_1,\cdots,i_l)}\big)(i_l)=(i_l)-(i_1,\cdots,i_l),$$ where, on the right-hand side, $(i_1,\cdots,i_l)$ now denotes the corresponding commutator in $\mathcal V$.
\end{lem}

\begin{proof}
Consider the product $(i_1,\cdots,i_l)(i_d,\infty)(i_1,\cdots,i_l)^{-1}$ and re-express it with only comb-claspers with $\infty$ in their support. To do so, as illustrated in Figure \ref{figuresondage}, we apply move $(2)$ from Proposition \ref{lemclaspcalculus} on the leaves on the $i_d$-th component, which introduces the comb-clasper $(i_1,\cdots,i_l,\infty)^{-1}$, and we simplify $(i_1,\cdots,i_l)$ and $(i_1,\cdots,i_l)^{-1}$.
\begin{figure}[!htbp]
    \centering
    \includegraphics[scale=0.7]{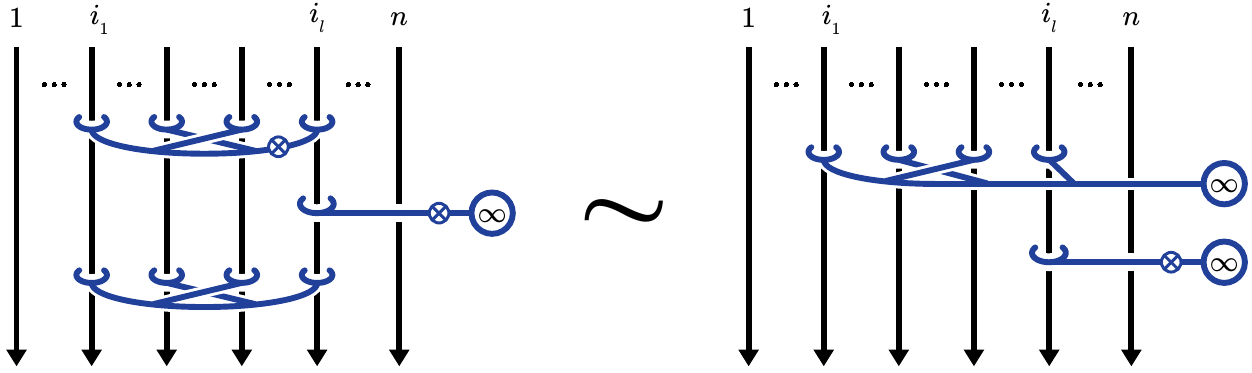}
    \caption{Proof of Lemma \ref{lemsondage}.}
    \label{figuresondage}
\end{figure}
\end{proof}
We can now state the injectivity of the representation $\gamma$.
\begin{thm}\label{thminj}
The representation $\gamma:\tilde B_n\mapsto GL(\mathcal{V})$ is injective.
\end{thm}
\begin{proof}

Let $\beta\in\tilde B_n$ be such that $\gamma(\beta)=Id$. First, Proposition \ref{propgammamatrix} imposes that $\beta$ is a pure braid; indeed the block $B_{1,1}$ must be the identity, which means that the permutation $\pi(\beta)$ is trivial.

According to Theorem \ref{normalformex} we can consider a normal form for $\beta$:
$$\beta=\prod(\alpha)^{\nu_{\alpha}}.$$

Let $I\subset\{1,\ \ldots,\ n\}$ be a sequence of indices with largest index $m$. Let also $\mathcal{V}_I$ be the subspace of $\mathcal V$ spanned by commutators with support included in $I$. We can then define the associated projection $p_I:\mathcal V \to\mathcal V_I$, and its composition with the restriction of $\gamma$ on $\mathcal V_I$, denoted by $\gamma_I:=p_I\circ\gamma_{|_{\mathcal V_I}}$. Note that it corresponds to keeping only the components with index in $I$.
It is clear using Proposition \ref{lemclaspcalculus} that $\gamma(\tilde P_n)(\mathcal{V}\setminus\mathcal V_I)\subset\mathcal{V}\setminus\mathcal V_I$, thus for $\beta_1,\ \beta_2\in\tilde P_n$ we have that $\gamma_I(\beta_1\beta_2)=\gamma_I(\beta_1)\gamma_I(\beta_2)$. 
Moreover $\gamma_I(\mathbf1^{(\alpha)})=Id$ for any comb-clasper $(\alpha)$ with $\supp(\alpha)\not\subset I$. Hence $\gamma_I(\beta)=\gamma_I(\beta')$ for $\beta'$ defined by:
$$\beta'=\prod_{\supp(\alpha)\subset I}(\alpha)^{\nu_{\alpha}}.$$
 
Now we show by strong induction on the degree of $(\alpha)$ that $\nu_\alpha=0$. For the base case we consider $I$ of the form $I=\{i,\ m\}$. Using Lemma \ref{lemsondage} we obtain:
\begin{align*}
    \gamma_I(\beta')(m)&=\gamma_I\big(\mathbf1^{(im)^{\nu_{im}}}\big)(m),\\
                    &=(m)-\nu_{im}\cdot(im).
\end{align*}
Because $\beta\in$ Ker$(\gamma)$, we have that $\gamma_I(\beta)(m)=(m)$, and this implies that $\nu_{\alpha}=0$ for any $(\alpha)$ of degree one. To prove that $\nu_{\alpha}=0$ for any $(\alpha)$ of degree $k$ we take $I$ of length $k+1$ and using the induction hypothesis, we get then:
$$\beta'=\prod_{\supp(\alpha)= I}(\alpha)^{\nu_{\alpha}}.$$
Thus thanks to Lemma \ref{lemsondage} we finally obtain: 
$$\gamma_I(\beta')(m)=(m)-\sum_{\supp(\alpha)= I}\nu_\alpha\cdot(\alpha).$$
Because $\beta\in$ Ker$(\gamma)$ we have that $\gamma_I(\beta)(m)=(m)$, and this implies $\nu_{\alpha}=0$ for any $(\alpha)$ of support $I$. Repeating the argument for any $I\subset\{1,\ \ldots,\ n\}$ of length $k+1$, we get that $\nu_{\alpha}=0$ for any $(\alpha)$ of degree $k$, which concludes the proof.
\end{proof}

\begin{cor}\label{normalformunicity}
The normal form is unique in $\tilde B_n$, i.e. if $\beta=\prod(\alpha)^{\nu_{\alpha}}=\prod(\alpha)^{\nu'_{\alpha}}$ are two normal forms of $\beta$ for a given order on the set of twisted comb-claspers, then $\nu_{\alpha}=\nu'_{\alpha}$ for any $(\alpha)$.
\end{cor}
\begin{proof}
The proof follows closely the previous one. As before for a given $I\subset\{1,\ \ldots,\ n\}$ we have $\gamma_I(\beta)=\gamma_I(\beta')$ for $\beta'$ defined by :
$$ \beta'=\prod_{\supp(\alpha)\subset I}(\alpha)^{\nu_{\alpha}}=\prod_{\supp(\alpha)\subset I}(\alpha)^{\nu'_{\alpha}}.$$
We show again by strong induction on the degree that $\nu_{\alpha}=\nu'_{\alpha}$. The base case is strictly similar, but for the inductive step one cannot in general write $\beta'$ with only comb-claspers with support $I$. However by Proposition \ref{lemclaspcalculus} a comb-clasper $(\alpha)$ with $\supp(\alpha)= I$ commutes with any comb-clasper $(\alpha')$ up to comb-claspers with support not included in $I$. Hence $\gamma_I(\mathbf1^{(\alpha)})$ commutes with $\gamma_I(\mathbf1^{(\alpha')})$ for any two comb-claspers $(\alpha')$ and $(\alpha)$ such that $\supp(\alpha)= I$. In particular we get:
\begin{align*}
    \gamma_I(\beta')(m)=&\gamma_I\left(\prod_{\supp(\alpha)\subsetneq I}(\alpha)^{\nu_\alpha}\right)\circ\gamma_I\left(\prod_{\supp(\alpha)= I}(\alpha)^{\nu_\alpha}\right)(m)\\
    =&\gamma_I\left(\prod_{\supp(\alpha)\subsetneq I}(\alpha)^{\nu'_\alpha}\right)\circ\gamma_I\left(\prod_{\supp(\alpha)= I}(\alpha)^{\nu'_\alpha}\right)(m).
\end{align*}
Since comb-claspers $(\alpha)$ with $\supp(\alpha)\subsetneq I$ have degree $<k-1$ where $k$ is the length of $I$, by induction hypothesis we can simplify the first factor in each expression. By Lemma \ref{lemsondage} we compute the second term thus obtaining:
$$(m)-\sum_{\supp(\alpha)=I}\nu_\alpha\cdot(\alpha)\\
    =(m)-\sum_{\supp(\alpha)=I}\nu'_\alpha\cdot(\alpha),$$
and the proof is complete.
\end{proof}
\begin{rmq}\label{ccadeaucamfaitplaisir}
Corollary \ref{normalformunicity} shows that the numbers $\nu_{\alpha}$ of parallel copies of each comb-clasper in a normal form are a complete invariant of pure braids up to link-homotopy. We call those numbers the \emph{clasp-numbers}. Other well known complete homotopy braid invariants are the Milnor numbers \cite{HabeggerLinHomotopy}. As a matter of fact, Milnor numbers can be used, using the techniques of \cite{YasuharaAkiraSelfdelteq}, to give another proof of Corollary \ref{normalformunicity}. In this paper we will not try to make explicit the relation between clasp-numbers and Milnor numbers, since we work solely with clasp-numbers. 
\end{rmq} 

\section{Links up to link-homotopy}\label{sectionlinks}

In the following of the paper we will focus on the study of $links$ up to link-homotopy. More precisely we will describe in terms of \emph{clasp-numbers variation} when two normal forms have link-homotopic $closures$.

The main purpose of this section is to use clasp-numbers, defined in Remark \ref{ccadeaucamfaitplaisir} above, to provide an explicit classification of links up to link-homotopy. In this way we recover results of Milnor \cite{MilnorLinkgrp} and Levine \cite{Levine4comp} for 4 or less components, and extend them partially for 5 components. To do so we first revisit in terms of claspers the work of Habegger and Lin \cite{HabeggerLinHomotopy}.
\begin{rmq}
Kotorii and Mizusawa also considered in \cite{kotorii_link-homotopy_2020} the question of using clasper theory to classify  4-component links up to link-homotopy. They use a different kind of normal form, arranged along a tetrahedron shape, adapted to the 4-component case. The main difference with the present work, however, is that their result makes direct use of Levine's classification. Here we instead reprove the latter using Theorem \ref{thmHabLin} and
clasper calculus. Our approach is likely to extend to the general case: as an illustration of this fact, we treat the algebraically split 5-component case at the end of this section.
\end{rmq}

\subsection{Habegger--Lin's work revisited}

There is a procedure on braids called \emph{closure}, that turns a braid into a link in $S^3$. The question is to determine when two braids have link-homotopic closures. Let us  first recall from \cite[Theorem 1.7 $\&$ Corollary 1.11]{HabeggerLinHomotopy} that for any integer $n$ we have the decomposition: $$\tilde P_n=\tilde P_{n-1}\ltimes\mathcal RF_{n-1}$$ where the first term corresponds to the braid obtained by omitting a given component, and the second term is the class of this component as an element of the reduced fundamental group of the disk with $n-1$ punctures.

To answer the question, Habegger and Lin in \cite{HabeggerLinHomotopy} study an action of $\tilde P_{2n}$ on $\tilde P_{n-1}\ltimes\mathcal RF_{n-1}$, which leads them to considering certain elementary operations $(\bar{x_i},\bar{x_i})_k$, $({x_i},{x_i})_k$ and $(\bar{x_i},{x_i})_k$, whose definition we recall here in terms of claspers.
\begin{defn}\label{defoperationelmt}
Let $\beta\in\tilde P_n$ be a pure homotopy braid, and let $i,\ k$ be two \emph{distinct} integers in $\{1,\ \ldots,\ n\}$.
\begin{itemize}
    \item $(\bar{x_i},\bar{x_i})_k(\beta)$ is the pure homotopy braid $\beta^\Delta\cdot\mathbf1^{(ik)^{-1}}$, where $\Delta$ and $(ik)^{-1}$ are degree one claspers as shown in the left-hand side of Figure \ref{figoperationclasp}.
    \item $({x_i},{x_i})_k(\beta)$ is the pure homotopy braid $\mathbf1^{(ik)}\cdot\beta^{\Delta'}$, where $\Delta'$ and $(ik)^{-1}$ are degree one claspers as shown in the central part of Figure \ref{figoperationclasp}.
    \item $(\bar{x_i},{x_i})_k(\beta)$ is the pure homotopy braid $1^{(ik)}\beta\cdot\mathbf1^{(ik)^{-1}}$, where $(ik)$ and $(ik)^{-1}$ are degree one claspers as shown in the right-hand side of Figure \ref{figoperationclasp}.
\end{itemize}
\begin{figure}[!htbp]
    \centering
    \includegraphics[width=\linewidth]{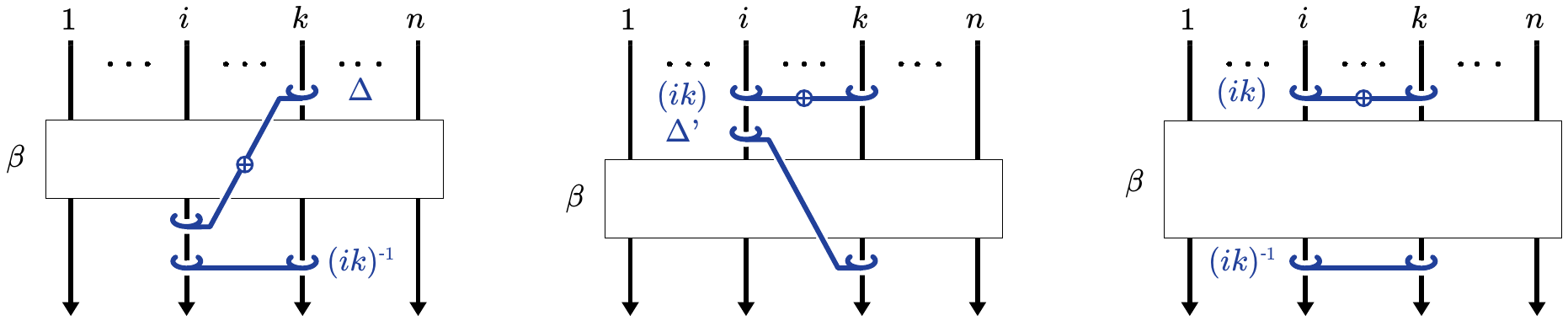}
    \caption{The elementary operations $(\bar{x_i},\bar{x_i})_k$, $(x_i,x_i)_k$, and $(\bar{x_i},x_i)_k$.}
    \label{figoperationclasp}
\end{figure}
\end{defn}
\begin{rmq}
In fact, in \cite{HabeggerLinHomotopy} those operations are only defined for $k=n$, but the definitions extend naturally for any $k\neq i$. Moreover, Figure 2.8 in \cite{HabeggerLinHomotopy} does not correspond exactly to Figure \ref{figoperationclasp}, due to convention choices. Firstly, in \cite{HabeggerLinHomotopy} braids are oriented from bottom to top whereas we orient them from top to bottom. Secondly, here the basepoint of the second term in the decomposition $\tilde P_n=\tilde P_{n-1}\ltimes\mathcal RF_{n-1}$ is taken above the $n-1$ punctures, and not under the $n-1$ punctures as in \cite{HabeggerLinHomotopy}.
\end{rmq}
We state now the main classification theorem of links up to link-homotopy.
\begin{thm}\label{thmHabLin}\cite{HabeggerLinHomotopy,Hughespartial} Let $\beta,\ \beta'\in\tilde{P}_n$ be two pure homotopy braids. The closures of $\beta$ and $\beta'$ are link-homotopic, if and only if there exists a sequence $\beta=\beta_0,\ \beta_1,\ \ldots,\ \beta_n=\beta'$ of elements of $\tilde P_n$ such that $\beta_{j+1}=(\bar{x_i},\bar{x_i})_k(\beta_j)$ for some $i\neq k$ in $\{1,\ \ldots,\ n\}$.
\end{thm}
\begin{proof}
Firstly, \cite[Lemma 2.11]{HabeggerLinHomotopy} and the proof of \cite[Theorem 2.13]{HabeggerLinHomotopy} imply that two pure homotopy braids whose closures are link-homotopic are related by a sequence of operations $(\bar{x_i},\bar{x_i})_k$, $({x_i},{x_i})_k$ and $(\bar{x_i},{x_i})_k$. Moreover, computations in \cite[pp.~413]{HabeggerLinHomotopy} show that operations $(\bar{x_i},\bar{x_i})_k$ generate operations $({x_i},{x_i})_k$. Finally, Hughes in \cite{Hughespartial} showed that operations $(\bar{x_i},{x_i})_k$ are also realized by operations $(\bar{x_i},\bar{x_i})_k$.
\end{proof}

\subsection{Link-homotopy classification of links with a small number of components.}

This section is dedicated to the explicit classification of links up to link-homotopy. The starting point of the strategy is Theorem \ref{thmHabLin} which allows us to see links up to link-homotopy as pure homotopy braids up to operations $(\bar{x_i},\bar{x_i})_k$ with $i\neq k$ in $\{1,\ \ldots,\ n\}$. Moreover with Corollary \ref{normalformunicity} we show that a braid is uniquely determined by its normal form, encoded by a sequence of integers: the clasp-numbers. The goal is then to determine how the normal form, or equivalently the clasp-numbers, vary under operations $(\bar{x_i},\bar{x_i})_k$. By using clasper calculus, we recover in this way the link-homotopy classification results from Milnor \cite{MilnorLinkgrp} and Levine \cite{Levine4comp} in the case of links with at most 4 components. We then apply these techniques to the 5-component \emph{algebraically split} case.

In order to use Corollary \ref{normalformunicity}, we need to fix an order on the set of twisted comb-claspers. In the rest of the paper we fix the following order, which is inspired from Example \ref{exnormalformcommu}. For two twisted comb-claspers $(\alpha)=(i_1\cdots i_l)$ and $(\alpha')=(i'_1\cdots i'_{l'})$ we set $(\alpha)\leq(\alpha')$ if:
\begin{itemize}
    \item $\degg(\alpha)<\degg(\alpha'),$ or
    \item $\degg(\alpha)=\degg(\alpha')$ and $i_1\ldots i_l<_{\lex}i'_1\ldots i'_l.$
\end{itemize} 
This order is used implicitly throughout the rest of the paper.

\subsubsection{The 3-component case.}

Let $L$ be a $3$-component link, then $L$ can be seen as the closure of a $3$-component string link $\beta$. As mentioned in Remark \ref{rmqstringlink}, up to link-homotopy, string links correspond to pure braids. Thus $\beta$ can be seen as the closure of the normal form: 
\[(12)^{\nu_{12}}(13)^{\nu_{13}}(23)^{\nu_{23}}(123)^{\nu_{123}},\]
for some integers $\nu_{12}$,\ $\nu_{13},\ \nu_{23}$ and $\nu_{123}$. See the left-hand side of Figure \ref{formenormal(123)}.
\begin{figure}[!htbp]
    \centering
    \includegraphics[width=\linewidth]{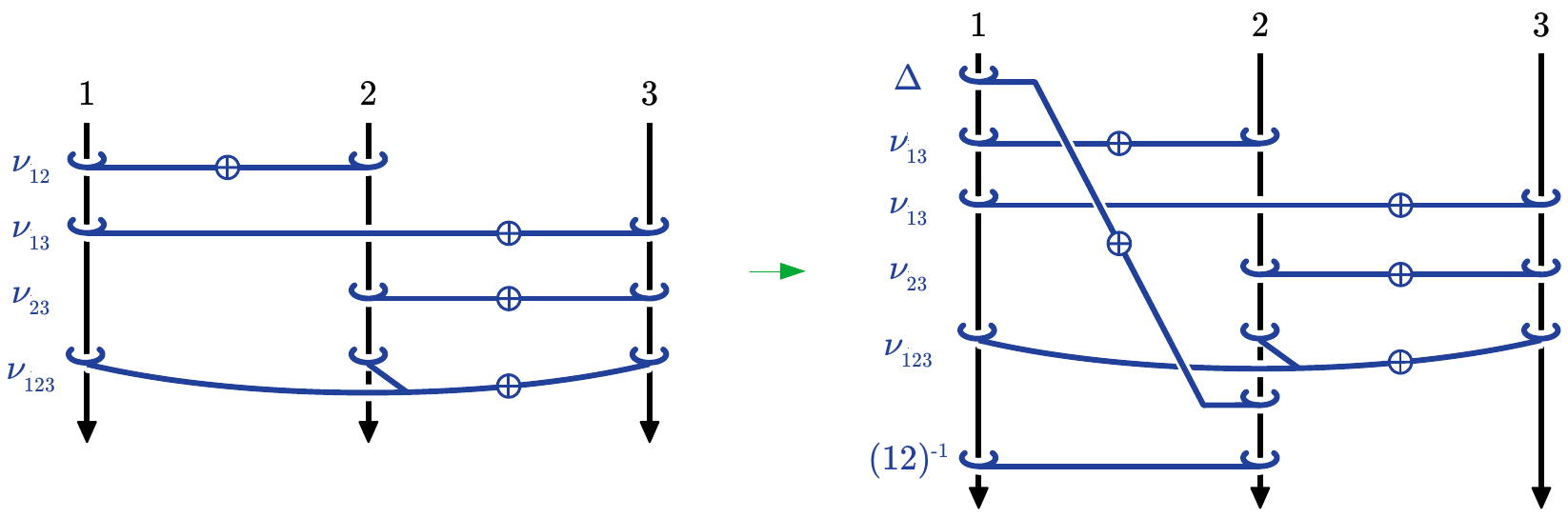}
    \caption{ Operation $(\bar{x_2},\bar{x_2})_1$ on the 3-component normal form.}
    \label{formenormal(123)}
\end{figure}

We now investigate how these numbers vary under operations $(\bar{x_i},\bar{x_i})_k$ for $i\neq k\in\{1,\ 2,\ 3\}$; we apply for example $(\bar{x_2},\bar{x_2})_1$. By Definition \ref{defoperationelmt} this corresponds to introducing the claspers $\Delta$ and $(12)^{-1}$ as shown in the right-hand side of Figure \ref{formenormal(123)}, which we then put in normal form. This is done by sliding the $1$-leaf of $\Delta$ along the first component to obtain $(12)$ and simplify it with $(12)^{-1}$. By move $(2)$ from Proposition \ref{lemclaspcalculus}, this sliding creates new claspers, but by Lemma \ref{lemfeuilledouble}, the only claspers that do not vanish up to link-homotopy, are those created when $\Delta$ crosses the leaves of $(13)^{\nu_{13}}$: more precisely, in this process, $\nu_{13}$ copies of $\{1,\ 2,\ 3\}$-supported claspers appear. Finally, according to Remark \ref{rmqclaspcalculus} we can rearrange these new claspers and the normal form becomes $$(12)^{\nu_{12}}(13)^{\nu_{13}}(23)^{\nu_{23}}(123)^{\nu_{123}+\nu_{13}}.$$
The other operations $(\bar{x_i},\bar{x_i})_k$ act in a similar way, by changing $\nu_{123}$ by a multiple of $\nu_{12}$, $\nu_{13}$ or $\nu_{23}$. 
Summarizing we have shown that $$\nu_{12},\ \nu_{13},\ \nu_{23} \text{ and } \nu_{123} \text{ mod} \text{ gcd}(\nu_{12},\ \nu_{13},\ \nu_{23}),$$ form a set of complete invariants for 3-component links up to link-homotopy. 

Note that we recover here Milnor invariants $\overline\mu_{12}$, $\overline\mu_{13}$, $\overline\mu_{23}$ and $\overline\mu_{123}$, that we already knew to be complete link-homotopy invariants for $3$-component links (see \cite{MilnorLinkgrp}). 

\subsubsection{The 4-component case.} 

Before proceeding with the link-homotopy classification of 4-component links, we need the following technical result. 
\begin{lem}\label{lembypass}
Let $C$ be a union of simple claspers for the trivial $n$-component braid $\mathbf1$, and let $l\in\{1,\ \ldots,\ n\}$. Let $T$ be a clasper in $C$ with $l$ in its support and let $C_T=\bigcup T'$ be the union of all claspers in $C$ such that $\supp(T')\cap\supp(T)=\{l\}$. Suppose that an $l$-leaf $ f$ of $T$ is disjoint from a $3$-ball $B$ containing all $l$-leaves of $C_T$. Then the closure of $\mathbf1^C$ is link-homotopic to the closure of $\mathbf 1^{C'}$ where $C'$ is obtained from $C$ by passing $f$ across the ball $B$ as shown in Figure \ref{figbypass}.
\end{lem}
\begin{proof}
First the result is clear if $T$ has several $l$-leaves, since by Lemma \ref{lemfeuilledouble}, $T$ vanishes up to link-homotopy. By Remark \ref{rmqclaspcalculus} the edges of any clasper in $C_T$ can freely cross those of $T$ but $f$ and the $l$-leaves of claspers in $C_T$ cannot be freely exchanged. However according to Remark \ref{rmqclaspcalculus} again, the leaf $f$ can be freely exchanged with any $l$-leaf of claspers in $C\setminus C_T$, since their supports contain at least 
some $k\neq l$ which is in $\supp(T)$. 
By using the closure we can thus slide $f$ in the other direction, using the closure of $\mathbf1$,  and bypass the $l$-leaves of claspers in $C_T$ all gathered in $B$. 
\end{proof}

\begin{figure}[!htbp]
    \centering
    \includegraphics[scale=0.8]{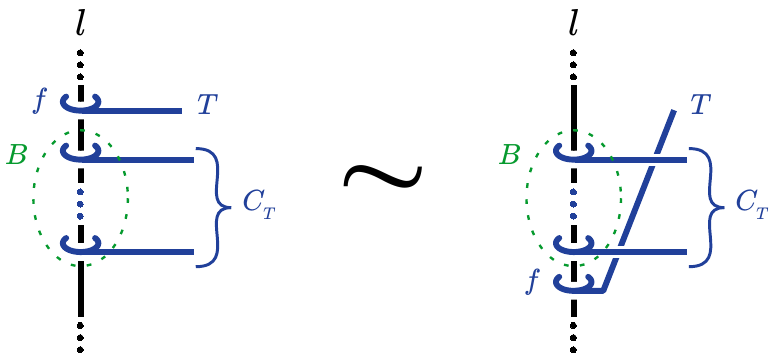}
    \caption{Illustration of Lemma \ref{lembypass}}
    \label{figbypass}
\end{figure}

Although the assumption of Lemma \ref{lembypass} may seem restrictive, it turns out to be naturally satisfied for normal forms. For instance, we have the following consequence. 
\begin{prop}\label{propbypass}
Let $C=(\alpha_1)^{\nu_1}\cdots(\alpha_m)^{\nu_m}$ be the normal form of a pure homotopy $n$-component braid and let $(\alpha)$ be a degree $n-2$ comb-clasper. Then $C$ and $C'=(\alpha_1)^{\nu_1}\cdots (\alpha)(\alpha_i)^{\nu_i}(\alpha)^{-1}\cdots(\alpha_m)^{\nu_m}$ have link-homotopic closures, for any $i\in\{1,\ \ldots,\ m\}$.
\end{prop}
\begin{proof}
We first consider the product $(\alpha_1)^{\nu_1}\cdots (\alpha_i)^{\nu_i}(\alpha)(\alpha)^{-1}\cdots(\alpha_m)^{\nu_m}$ where we just insert the trivial term $(\alpha)(\alpha)^{-1}$ in $C$. We next want to exchange $(\alpha)$ and $(\alpha_i)^{\nu_i}$. This is allowed if $|\supp(\alpha)\cap\supp(\alpha_i)|\geq2$ by Remark \ref{rmqclaspcalculus}, but if $\supp(\alpha)\cap\supp(\alpha_i)=\{l\}$ we can only realize crossing changes between the edges of $(\alpha)$ and $(\alpha_i)^{\nu_i}$ (see Remark \ref{rmqclaspcalculus}). However in that case $(\alpha_i)$ is a comb-clasper of support $\{k,\ l\}$ with $k$ the only component not in the support of $(\alpha)$, thus we can apply Lemma \ref{lembypass} to the $l$-leaf of $(\alpha)$, and bypass the block $(\alpha_i)^{\nu_i}$ (corresponding to $C_T$ in Lemma \ref{lembypass}).
\end{proof}

\begin{figure}[!htbp]
    \centering    
    \includegraphics[scale=0.6]{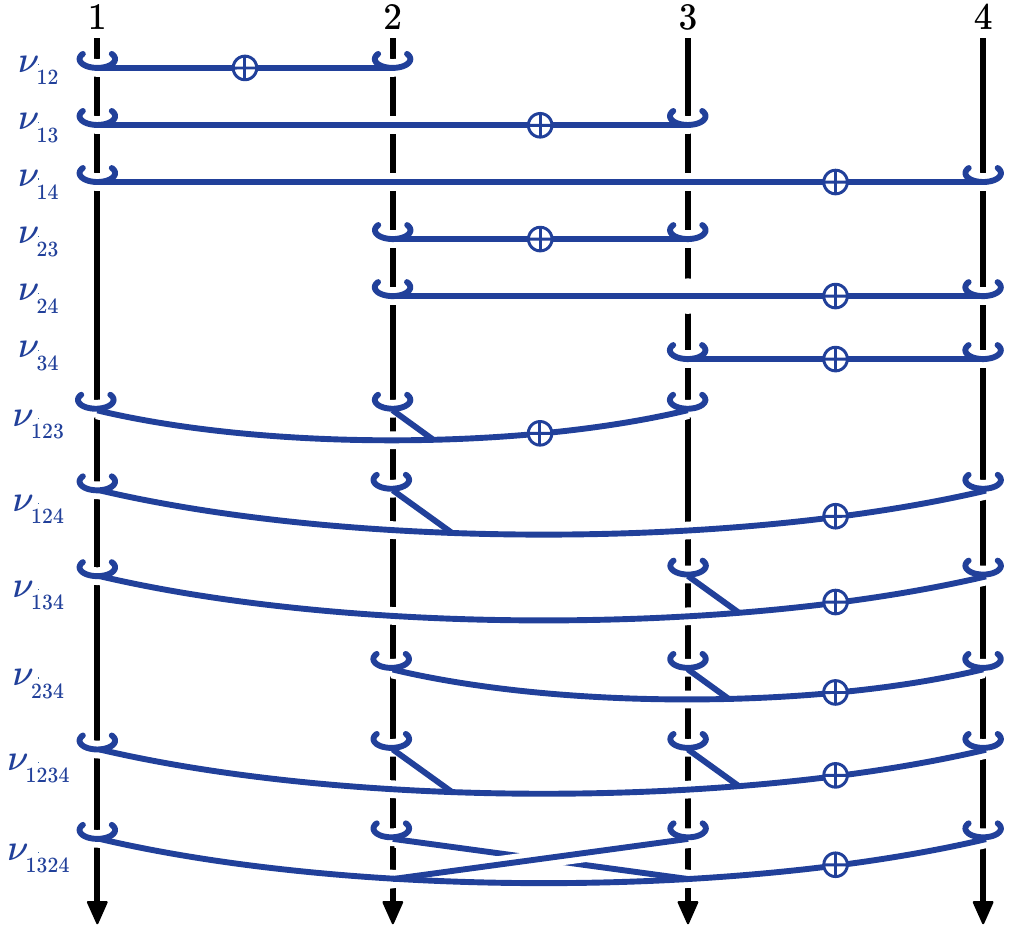}
    \caption{Normal form for 4 components.}
    \label{formenormal(1234)}
\end{figure}
Let us now return to the classification of links up to link-homotopy and let $L$ be a $4$-component link seen as the closure of the normal form:  
$$(12)^{\nu_{12}}(13)^{\nu_{13}}(14)^{\nu_{14}}(23)^{\nu_{23}}(24)^{\nu_{24}}(34)^{\nu_{34}}(123)^{\nu_{123}}(124)^{\nu_{124}}(134)^{\nu_{134}}(234)^{\nu_{234}}(1234)^{\nu_{1234}}(1324)^{\nu_{1324}},$$
for some integers $\nu_{12}$, $\nu_{13},$ $\nu_{14},$ $\nu_{23},$ $\nu_{24},$ $\nu_{34},$ $\nu_{123},$ $\nu_{124},$ $\nu_{134},$ $\nu_{234},$ $\nu_{1234},$ and $\nu_{1324}$.
See Figure \ref{formenormal(1234)}.

We can apply Proposition \ref{propbypass} to the degree $2$ comb-claspers $(123)$, $(124)$, $(134)$ and $(234)$. For example, applying Proposition \ref{propbypass} to $(\alpha)=(234)$ and $(\alpha_i)=(12)$, we get that $L$ is link-homotopic to the closure of:
\begin{align*}
   &{\color{red}(234)}(12)^{\nu_{12}}{\color{red}(234)^{-1}}(13)^{\nu_{13}}(14)^{\nu_{14}}(23)^{\nu_{23}}(24)^{\nu_{24}}(34)^{\nu_{34}}(123)^{\nu_{123}}\\&(124)^{\nu_{124}}(134)^{\nu_{134}}(234)^{\nu_{234}}(1234)^{\nu_{1234}}(1324)^{\nu_{1324}}.
\end{align*}
By clasper calculus (Proposition \ref{lemclaspcalculus} and Remark \ref{rmqclaspcalculus}), we have $(234)(12)^{\nu_{12}}(234)^{-1}\sim(12)^{\nu_{12}}(1234)^{\nu_{12}}$. The product of claspers $(1234)^{\nu_{12}}$ can be freely homotoped by Remark \ref{rmqclaspcalculus}, thus producing the normal form 
$$(12)^{\nu_{12}}(13)^{\nu_{13}}(14)^{\nu_{14}}(23)^{\nu_{23}}(24)^{\nu_{24}}(34)^{\nu_{34}}(123)^{\nu_{123}}(124)^{\nu_{124}}(134)^{\nu_{134}}(234)^{\nu_{234}}(1234)^{\nu_{1234}{\color{red}+\nu_{12}}}(1324)^{\nu_{1324}},$$
whose closure is link-homotopic to $L$.
This is recorded in the first row of Table \ref{tableau1}, which records all possible transformations on clasp-numbers obtained with Proposition \ref{propbypass}.  Each row represents a possible transformation where the entry in the column $\nu_{\alpha}$ represents the variation of the clasp-number $\nu_{\alpha}$. Note that an empty cell means that the corresponding clasp-number remains unchanged. Note also that, we only need two columns because for the comb-claspers of degree $1$ or $2$ the associated clasp-numbers remain unchanged.
\begin{table}[!htbp]
\centering
\begin{tabular}{|c|c|}
\hline
$\nu_{1234}$&$\nu_{1324}$\\
\hline\hline
$\nu_{12}$&\\
\hline
$\nu_{34}$&\\
\hline
&$\nu_{13}$\\
\hline
&$\nu_{24}$\\
\hline
$\nu_{14}$&-$\nu_{14}$\\
\hline
$\nu_{23}$&-$\nu_{23}$\\
\hline
\end{tabular}
\vspace{0.08cm}
\caption{Some clasp-numbers variation with same closures.}
\label{tableau1}
\end{table}

Let us now describe how operations $(\bar{x_i},\bar{x_i})_k$ for $i\neq k$ in $\{1,\ \ldots,\ 4\}$ affect clasp-numbers. As for the $3$-component case, $(\bar{x_i},\bar{x_i})_k$ corresponds to sliding the $i$-leaf of a simple clasper of support $\{i,\ j\}$ (denoted $\Delta$ in Definition \ref{defoperationelmt}) along the $i$-th component. Along the way $\Delta$ encounters leaves and edges of other claspers, that can be crossed as described by moves $(2)$ and $(4)$ of Proposition \ref{lemclaspcalculus}. In doing so, claspers of degree $2$ and $3$ may appear, that we must reposition in the normal form. Those of degree $3$ commute with any clasper by Remark \ref{rmqclaspcalculus}, but since they may not be comb-claspers we have to use IHX relations (Proposition \ref{lemIHX}) to turn them into comb-claspers. Claspers of degree $2$ can be repositioned using Remark \ref{rmqclaspcalculus} and Lemma \ref{lembypass} (the fact that Lemma \ref{lembypass} applies is clear according to the shape of the normal form, where factors are stacked).

We detail as an example operation $(\bar{x_4},\bar{x_4})_2$. In that case $\Delta$ has support $\{2,\ 4\}$ and we slide its $2$-leaf along the $2$nd component. According to Remark \ref{rmqclaspcalculus}, $\Delta$ can freely cross the edges of claspers with $4$ in their support and the $2$-leaves of claspers containing $2$ and $4$ in their support. Thus we only consider the claspers that appear when $\Delta$ meets the edges of $(13)^{\nu_{13}}$ and the $2$-leaves of $(12)^{\nu_{12}},$ $(23)^{\nu_{23}}$ and  $(123)^{\nu_{123}}$. Once repositioned we obtain in order the factors $(1324)^{\nu_{13}}$, $(124)^{\nu_{12}}$, $(234)^{-\nu_{23}}$ and $(1324)^{-\nu_{123}}$. However according to Table \ref{tableau1}, $(1324)^{\nu_{13}}$ can be removed up to link-homotopy and thus we get the following normal form: 
\begin{align*}
&(12)^{\nu_{12}}(13)^{\nu_{13}}(14)^{\nu_{14}}(23)^{\nu_{23}}(24)^{\nu_{24}}(34)^{\nu_{34}}(123)^{\nu_{123}}(124)^{\nu_{124}{\color{red}+\nu_{12}}}\\&(134)^{\nu_{134}}(234)^{\nu_{234}{\color{red}-\nu_{23}}}(1234)^{\nu_{1234}}(1324)^{\nu_{1324}{\color{red}-\nu_{123}}}.
\end{align*}

In the same way, we compute all operations $(\bar{x_i},\bar{x_i})_k$ and record them in Table \ref{tableau}. The entry in row $(\bar{x_i},\bar{x_i})_k$ represents the corresponding operation. As in Table \ref{tableau1}, an empty cell means that $(\bar{x_i},\bar{x_i})_k$ does not change the clasp-number. Moreover the $\nu_{ik}$ columns are omitted because they remain unchanged by any operations. 

\

\begin{table}[!htbp]
\centering
\begin{tabular}{|c||c|c|c|c|c|c|}
\hline
&$\nu_{123}$&$\nu_{124}$&$\nu_{134}$&$\nu_{234}$&$\nu_{1234}$&$\nu_{1324}$\\
\hline\hline
$(\bar{x_2},\bar{x_2})_1$&$\nu_{13}$&$\nu_{14}$&&&$\nu_{134}$&\\
\hline
$(\bar{x_3},\bar{x_3})_1$&$-\nu_{12}$&&$\nu_{14}$&&&$\nu_{124}$\\
\hline
$(\bar{x_4},\bar{x_4})_1$&&$-\nu_{12}$&$-\nu_{13}$&&$-\nu_{123}$&$\nu_{123}$\\
\hline\hline
$(\bar{x_1},\bar{x_1})_2$&$-\nu_{23}$&$-\nu_{24}$&&&$-\nu_{234}$&\\
\hline
$(\bar{x_3},\bar{x_3})_2$&$\nu_{12}$&&&$\nu_{24}$&$\nu_{124}$&$-\nu_{124}$\\
\hline
$(\bar{x_4},\bar{x_4})_2$&&$\nu_{12}$&&$-\nu_{23}$&&$-\nu_{123}$\\
\hline\hline
$(\bar{x_1},\bar{x_1})_3$&$\nu_{23}$&&$-\nu_{34}$&&&$\nu_{234}$\\
\hline
$(\bar{x_2},\bar{x_2})_3$&$-\nu_{13}$&&&$-\nu_{34}$&$-\nu_{134}$&$\nu_{134}$\\
\hline
$(\bar{x_4},\bar{x_4})_3$&&&$\nu_{13}$&$\nu_{23}$&$\nu_{123}$&\\
\hline\hline
$(\bar{x_1},\bar{x_1})_4$&&$\nu_{24}$&$\nu_{34}$&&$\nu_{234}$&$-\nu_{234}$\\
\hline
$(\bar{x_2},\bar{x_2})_4$&&$-\nu_{14}$&&$\nu_{34}$&&$-\nu_{134}$\\
\hline
$(\bar{x_3},\bar{x_3})_4$&&&$-\nu_{14}$&$-\nu_{24}$&$-\nu_{124}$&\\
\hline
\end{tabular}
\vspace{0.08cm}
\caption{Clasp-numbers variations under operations $(\bar{x_i},\bar{x_i})_k$.}
\label{tableau}
\end{table}

There are however algebraic redundancies in Table \ref{tableau}, i.e. some lines are combinations of other lines, which means that some operation $(\bar{x_i},\bar{x_i})_k$ generate the others. So we can keep only these ones (or their opposite), which we call \emph{“generating"} operations, and which we record in Table \ref{tableau'}.

\begin{table}[!htbp]
\centering
\begin{tabular}{|c|c|c|c|c|c|}
\hline
$\nu_{123}$&$\nu_{124}$&$\nu_{134}$&$\nu_{234}$&$\nu_{1234}$&$\nu_{1324}$\\
\hline
\hline
$\nu_{13}$&$\nu_{14}$&&&$\nu_{134}$&\\
\hline
$-\nu_{12}$&&$\nu_{14}$&&&$\nu_{124}$\\
\hline
$\nu_{23}$&$\nu_{24}$&&&$\nu_{234}$&\\
\hline
&$-\nu_{12}$&&$\nu_{23}$&&$\nu_{123}$\\
\hline
$\nu_{23}$&&$-\nu_{34}$&&&$\nu_{234}$\\
\hline
&&$\nu_{13}$&$\nu_{23}$&$\nu_{123}$&\\
\hline
&$\nu_{14}$&&$-\nu_{34}$&&$\nu_{134}$\\
\hline
&&$\nu_{14}$&$\nu_{24}$&$\nu_{124}$&\\
\hline
\end{tabular}
\vspace{0.08cm}
\caption{Clasp-numbers variations under generating operations.}
\label{tableau'}
\end{table}

Finally, with Table \ref{tableau'} we reinterpret the homotopy classification of 4-component links as follows.
\begin{thm}\label{thm4compclass}
Two 4-component links, seen as closures of braids in normal forms (see Figure \ref{formenormal(1234)}), are link-homotopic if and only if their clasp-numbers are related by a sequence of transformations from Table \ref{tableau'}.
\end{thm}
\begin{rmq}\label{rmqtable}
Table \ref{tableau1} was only used here as a tool to simplify the computations summarized in Table \ref{tableau}. We stress that Table \ref{tableau'} alone suffices to generate Table \ref{tableau1} and Table \ref{tableau}. In particular, Table \ref{tableau1} is obtained by “commuting" the rows of Table \ref{tableau'}. More precisely let us denote by $[R_i]_k$ the variation associated to the $i$-th row of Table $k$. Let us also denote by $[R_i,R_j]_k$ the “commutator of rows $i$ and $j$" from Table $k$, i.e. the variation obtained by applying the $i$-th row of Table $k$, then the $j$-th, then the opposite of the $i$-th and finally the opposite of the $j$-th. Thus, Table \ref{tableau'} generates the rows of Table \ref{tableau1} as follows:
$$\begin{array}{lcr}
[R_1]_{\ref{tableau1}}=[R_6,R_2]_{\ref{tableau'}},&\quad [R_2]_{\ref{tableau1}}=[R_1,R_5]_{\ref{tableau'}},\quad& [R_3]_{\ref{tableau1}}=[R_6,R_7]_{\ref{tableau'}},\\

[R_4]_{\ref{tableau1}}=[R_3,R_2]_{\ref{tableau'}},& [R_5]_{\ref{tableau1}}=[R_2,R_1]_{\ref{tableau'}},& [R_6]_{\ref{tableau1}}=[R_5,R_6]_{\ref{tableau'}}.\end{array}$$
\end{rmq}

Note that Levine in \cite{Levine4comp} already proved a similar result. The purpose of this paragraph is to explain the correspondence between the two approaches. The strategy adopted in \cite{Levine4comp} consists in fixing the first three components and let the fourth one carry the information of the link-homotopy indeterminacy. Levine used four integers $k,\ l,\ r,\ d$ to describe a normal form for the first three components, and integers $e_i$ with $\ i\in\{1,\ \ldots,\ 8\}$ to describe the information of the last component. Finally in \cite[Table3]{Levine4comp} he gives a list of all possible transformations on $e_i$-numbers that do not change the link-homotopy class. Fixing the last component corresponds in our setting to fixing the clasp-number $\nu_{123}$: this is why \cite[Table 3]{Levine4comp} has one less column than Tables \ref{tableau} and \ref{tableau'}. Moreover the five rows of \cite[Table 3]{Levine4comp} correspond to $(\bar{x_3},\bar{x_3})_1^{-1},(\bar{x_4},\bar{x_4})_2^{-1},(\bar{x_1},\bar{x_1})_4,\ (\bar{x_3},\bar{x_3})_4$ and $(\bar{x_1},\bar{x_1})_2^{-c}\circ(\bar{x_3},\bar{x_3})_1^{-a}\circ(\bar{x_2},\bar{x_2})_1^{-b}$, respectively, and Levine's integers correspond to clasp-numbers as follows.
\begin{table}[!htbp]
    \centering
    \begin{tabular}{|c|c|c|c|c|c|c|c|c|c|c|c|}
         \hline
         $k$ & $r$ & $l$ & $d$ & $e_1$ &  $e_2$ & $e_3$ & $e_4$ & $e_5$ & $e_6$ & $e_7$ & $e_8$ \\
         \hline
         $\nu_{12}$& $\nu_{13}$&$\nu_{23}$&$\nu_{123}$&$\nu_{14}$&$\nu_{24}$&$\nu_{34}$&$\nu_{124}$&$\nu_{134}$&$\nu_{234}$&$-\nu_{1324}$&$-\nu_{1234}$\\
         \hline
    \end{tabular}
    \label{tablelevine}
\end{table}
\subsubsection{The 5-component algebraically split case.}

This section is dedicated to the study of \emph{5-components algebraically split links}. These are links such that the linking number is zero for any pair of components. Equivalently, algebraically split links are given by the closure of a normal form with trivial clasp-numbers for any degree one comb-clasper.

The following proposition is the algebraically split version of Proposition \ref{propbypass}. The proof is essentially same and is left to the reader.

\begin{prop}\label{propbypass'}
Let $C=(\alpha_1)^{\nu_1}\cdots(\alpha_m)^{\nu_m}$ be a normal form of a pure homotopy $n$-component braid with $\nu_i=0$ for any $(\alpha_i)$ of degree one, and let $(\alpha)$ be a degree $n-3$ comb-clasper. Then $C$ and $C'=(\alpha_1)^{\nu_1}\cdots (\alpha)(\alpha_i)^{\nu_i}(\alpha)^{-1}\cdots(\alpha_m)^{\nu_m}$ have link-homotopic closures, for any $i\in\{1,\ \ldots,\ m\}$.
\end{prop}

Now, let $L$ be a $5$-component algebraically split link seen as the closure of the normal form: 
\begin{align*} C=&(123)^{\nu_{123}}(124)^{\nu_{124}}(125)^{\nu_{125}}(134)^{\nu_{134}}(135)^{\nu_{135}}(145)^{\nu_{145}}(234)^{\nu_{234}}(235)^{\nu_{235}}(245)^{\nu_{245}}(345)^{\nu_{345}}(1234)^{\nu_{1234}}\\&(1235)^{\nu_{1235}}(1245)^{\nu_{1245}}(1324)^{\nu_{1324}}(1325)^{\nu_{1325}}(1345)^{\nu_{1345}}(1425)^{\nu_{1425}}(1435)^{\nu_{1435}}(2345)^{\nu_{2345}}(2435)^{\nu_{2435}}\\& (12345)^{\nu_{12345}}(12435)^{\nu_{12435}}(13245)^{\nu_{13245}}(13425)^{\nu_{13425}}(14235)^{\nu_{14235}}(14325)^{\nu_{14325}}.
\end{align*}

The strategy is similar to the 4-component case. We see links as braid closures, and with Theorem \ref{normalformunicity} we know that any braid is uniquely determined by a set of clasp-numbers $\{\nu_\alpha\}$. In this case, the algebraically split condition results in the vanishing of clasp-numbers $\nu_{ij}$ (i.e. $\nu_\alpha=0$ for all $(\alpha)$ of degree 1). Now, as mentioned by Theorem \ref{thmHabLin}, the classification of links up to link-homotopy reduces to determining how operations $(\bar{x_i},\bar{x_i})_k$ for $i\neq k$ in $\{1,\ \ldots,\ 5\}$ affect the clasp-numbers. 

We first use Proposition \ref{propbypass'} to simplify the upcoming computations. In that case Proposition \ref{propbypass'} concerns degree 2 comb-claspers $(123)$, $(124)$, $(125)$, $(134)$, $(135)$, $(145)$, $(234)$, $(235)$, $(245)$ and $(345)$. We record in Table \ref{tablesometransfo'} all possible transformations on clasp-numbers obtained with Proposition \ref{propbypass'}. As before, each row represents a possible transformation, where the entry in the column $\nu_\alpha$ represents the variation of the clasp-number $\nu_\alpha$, and an empty cell means that the corresponding clasp-number remains unchanged. Note also that we only need columns corresponding to degree 4 comb-claspers because the other clasp-numbers remain unchanged.

\begin{table}[!htbp]
\centering

\begin{tabular}{|c|c|c|c|c|c|}

\hline
$\nu_{12345}$&$\nu_{12435}$&$\nu_{13245}$&$\nu_{13425}$&$\nu_{14235}$&$\nu_{14325}$\\
\hline
\hline
$\nu_{123}$&&&&&\\
\hline
&&$\nu_{123}$&&&\\
\hline
&$\nu_{124}$&&&&\\
\hline
&&&&$\nu_{124}$&\\
\hline
$\nu_{125}$&$-\nu_{125}$&&&&\\
\hline
&&&$\nu_{125}$&&$-\nu_{125}$\\
\hline
&&&$\nu_{134}$&&\\
\hline
&&&&&$\nu_{134}$\\
\hline
&$\nu_{135}$&&&$-\nu_{135}$&\\
\hline
&&$\nu_{135}$&$-\nu_{135}$&&\\
\hline
%
%
$\nu_{145}$&&$-\nu_{145}$&&&\\
\hline
&&&&$\nu_{145}$&$-\nu_{145}$\\
\hline
$\nu_{234}$&$-\nu_{234}$&&$-\nu_{234}$&&$\nu_{234}$\\
\hline
&$\nu_{234}$&$-\nu_{234}$&$\nu_{234}$&$-\nu_{234}$&\\
\hline
&&&&$\nu_{235}$&\\
\hline
&&&&&$\nu_{235}$\\
\hline
&&$\nu_{245}$&&&\\
\hline
&&&$\nu_{245}$&&\\
\hline
$\nu_{345}$&&&&&\\
\hline
&$\nu_{345}$&&&&\\
\hline
\end{tabular}
\caption{Some clasp-numbers variations with same closure.}
\label{tablesometransfo'}
\end{table}
Finally, we compute the effect of all operations $(\bar{x_i},\bar{x_i})_k$ using Definition \ref{defoperationelmt} and Table \ref{tablesometransfo'}, and simplify the results keeping only the \emph{“generating"} operations, as in the 4-component case. We record the corresponding clasp-number variations in Table \ref{tableclaspvariasimplifier'}. As for the $4$-component case, Table \ref{tableclaspvariasimplifier'} contains all the data for the classification of $5$-component algebraically split links. In other words we obtain the following classification result. 

\begin{thm}\label{thm5compclass}
Two 5-component algebraically split links, seen as closures of braids in normal forms, are link-homotopic if and only if their clasp-numbers are related by a sequence of transformations from Table \ref{tableclaspvariasimplifier'}.
\end{thm}

\begin{rmq}

Just as in Remark \ref{rmqtable}, only Table \ref{tableclaspvariasimplifier'} is needed here as it generates Table \ref{tablesometransfo'}. With the same notations as in Remark \ref{rmqtable} and with the additional notation “$\circ$" for composition, we get:
$$\begin{array}{cccc}
    [R_1]_{\ref{tablesometransfo'}}=[R_{12},R_3]_{\ref{tableclaspvariasimplifier'}}\circ[R_5,R_6]_{\ref{tableclaspvariasimplifier'}},& [R_2]_{\ref{tablesometransfo'}}=[R_{6},R_{5}]_{\ref{tableclaspvariasimplifier'}},& [R_3]_{\ref{tablesometransfo'}}=[R_{11},R_{12}]_{\ref{tableclaspvariasimplifier'}},&[R_4]_{\ref{tablesometransfo'}}=[R_6,R_{14}]_{\ref{tableclaspvariasimplifier'}},\\

[R_5]_{\ref{tablesometransfo'}}=[R_5,R_{11}]_{\ref{tableclaspvariasimplifier'}}\circ[R_3,R_{11}]_{\ref{tableclaspvariasimplifier'}},& [R_6]_{\ref{tablesometransfo'}}=[R_3,R_{11}]_{\ref{tableclaspvariasimplifier'}},&
[R_7]_{\ref{tablesometransfo'}}=[R_{12},R_{13}]_{\ref{tableclaspvariasimplifier'}},& [R_8]_{\ref{tablesometransfo'}}=[R_8,R_9]_{\ref{tableclaspvariasimplifier'}},\\

[R_9]_{\ref{tablesometransfo'}}=[R_{1},R_{5}]_{\ref{tableclaspvariasimplifier'}},& [R_{10}]_{\ref{tablesometransfo'}}=[R_{13},R_{5}]_{\ref{tableclaspvariasimplifier'}},& 
[R_{11}]_{\ref{tablesometransfo'}}=[R_{2},R_{1}]_{\ref{tableclaspvariasimplifier'}},& [R_{12}]_{\ref{tablesometransfo'}}=[R_{13},R_{14}]_{\ref{tableclaspvariasimplifier'}},\\

[R_{13}]_{\ref{tablesometransfo'}}=[R_{6},R_{4}]_{\ref{tableclaspvariasimplifier'}},& [R_{14}]_{\ref{tablesometransfo'}}=[R_{7},R_{9}]_{\ref{tableclaspvariasimplifier'}},& 
[R_{15}]_{\ref{tablesometransfo'}}=[R_{5},R_{10}]_{\ref{tableclaspvariasimplifier'}},& [R_{16}]_{\ref{tablesometransfo'}}=[R_{7},R_{3}]_{\ref{tableclaspvariasimplifier'}},\\

[R_{17}]_{\ref{tablesometransfo'}}=[R_{4},R_{2}]_{\ref{tableclaspvariasimplifier'}},& [R_{18}]_{\ref{tablesometransfo'}}=[R_{10},R_{11}]_{\ref{tableclaspvariasimplifier'}},& 
[R_{19}]_{\ref{tablesometransfo'}}=[R_{1},R_{7}]_{\ref{tableclaspvariasimplifier'}},& [R_{20}]_{\ref{tablesometransfo'}}=[R_{10},R_{1}]_{\ref{tableclaspvariasimplifier'}}.\end{array}$$
\end{rmq}
\begin{amssidewaystable}
\centering
\resizebox{\linewidth}{!}{
\begin{tabular}{|c|c|c|c|c|c|c|c|c|c|c|c|c|c|c|c|}
\hline
$\nu_{1234}$&$\nu_{1235}$&$\nu_{1245}$&$\nu_{1324}$&$\nu_{1325}$&$\nu_{1345}$&$\nu_{1425}$&$\nu_{1435}$&$\nu_{2345}$&$\nu_{2435}$&$\nu_{12345}$&$\nu_{12435}$&$\nu_{13245}$&$\nu_{13425}$&$\nu_{14235}$&$\nu_{14325}$\\
\hline
\hline
$\nu_{134}$&$\nu_{135}$&$\nu_{145}$&&&&&&&&$\nu_{1345}$&$\nu_{1435}$&&&&\\
\hline
&&&$\nu_{124}$&$\nu_{125}$&$\nu_{145}$&&&&&&&$\nu_{1245}$&$\nu_{1425}$&&\\
\hline
$-\nu_{123}$&&&$\nu_{123}$&&&$\nu_{125}$&$\nu_{135}$&&&&&&&$\nu_{1235}$&$\nu_{1325}$\\
\hline
$\nu_{234}$&$\nu_{235}$&$\nu_{245}$&&&&&&&&$\nu_{2345}$&$\nu_{2435}$&&&&\\
\hline
&&$\nu_{125}$&$-\nu_{123}$&&&$-\nu_{125}$&&&$\nu_{235}$&&$\nu_{1235}$&$\nu_{1325}$&$-\nu_{1325}$&$-\nu_{1235}$&\\
\hline
&&&&$\nu_{123}$&&$\nu_{124}$&&$\nu_{234}$&$-\nu_{234}$&&&&$\nu_{1234}+\nu_{1324}$&&$-\nu_{1234}$\\
\hline
&&&$\nu_{234}$&$\nu_{235}$&$-\nu_{345}$&&&&&&&$\nu_{2345}+\nu_{2435}$&$-\nu_{2435}$&&\\
\hline
$\nu_{134}$&$\nu_{135}$&&$-\nu_{134}$&$-\nu_{135}$&&&&$\nu_{345}$&&$\nu_{1345}$&&$-\nu_{1345}$&&$\nu_{1435}$&$-\nu_{1435}$\\
\hline
&$-\nu_{123}$&&&&&&$\nu_{134}$&&$\nu_{234}$&&$\nu_{1234}+\nu_{1324}$&&&$-\nu_{1324}$&\\
\hline
$\nu_{234}$&&&$-\nu_{234}$&&&$\nu_{245}$&$\nu_{345}$&&&&&&&$\nu_{2345}+\nu_{2435}$&$-\nu_{2345}$\\
\hline
$\nu_{124}$&&&&&$\nu_{145}$&&$-\nu_{145}$&$\nu_{245}$&$-\nu_{245}$&$\nu_{1245}$&$-\nu_{1245}$&&$\nu_{1425}$&&$-\nu_{1425}$\\
\hline
&&$\nu_{124}$&&&$\nu_{134}$&&&$\nu_{234}$&&$\nu_{1234}$&&$\nu_{1324}$&&&\\
\hline
&&&&$\nu_{135}$&&$\nu_{145}$&&$-\nu_{345}$&$\nu_{345}$&&&&$\nu_{1345}$&&$\nu_{1435}$\\
\hline
&$\nu_{125}$&&&&&&$\nu_{145}$&&$\nu_{245}$&&$\nu_{1245}$&&&$\nu_{1425}$&\\
\hline
&&$\nu_{125}$&&&$\nu_{135}$&&&$\nu_{235}$&&$\nu_{1235}$&&$\nu_{1325}$&&&\\
\hline
\end{tabular}}
\caption{Clasp-numbers variations under generating operations in the 5-component algebraically split case.}
\label{tableclaspvariasimplifier'}
\end{amssidewaystable}

\newpage

\bibliographystyle{plain}
\bibliography{bibli.bib}

\begin{thebibliography}{10}

\bibitem{ArtinBraid}
E.~Artin.
\newblock Theory of braids.
\newblock {\em Ann. Math. (2)}, 48:101--126, 1947.

\bibitem{AudouxJBWagnerCodim2EmbeddingHomo}
B.~Audoux, J-B. Meilhan, and E.~Wagner.
\newblock On codimension two embeddings up to link-homotopy.
\newblock {\em J. Topol.}, 10:1107--1123, 2017.

\bibitem{BartelsTeichner2DimLinksNullHomo}
A.~Bartels and P.~Teichner.
\newblock All two dimensions links are null homotopic.
\newblock {\em Geom. Topol.}, 3:235--252, 1999.

\bibitem{FennRolfsenSphereHomo4space}
R.~Fenn and D.~Rolfsen.
\newblock Spheres may link homotopically in 4-space.
\newblock {\em J. Lond. Math. Soc., II. Ser.}, 34:177--184, 1986.

\bibitem{Flemingyasuselfck}
T~Fleming and A.~Yasuhara.
\newblock Milnor's invariants and self {{\(C_{k}\)}}-equivalence.
\newblock {\em Proc. Am. Math. Soc.}, 137:761--770, 2009.

\bibitem{GoldsmithHomotopybraids}
D.~L. Goldsmith.
\newblock Homotopy of braids - in answer to a question of {E}. {Artin}.
\newblock Topology {Conf}, {Virginia} polytechnic {Inst}. and {State} {Univ}.
  1973, {Lect}. {Notes} {Math}. 375, 91-96., 1974.

\bibitem{siteweb}
E~Graff.
\newblock https://graff201.users.lmno.cnrs.fr, 2022.

\bibitem{HabeggerLinHomotopy}
N.~Habegger and X-S. Lin.
\newblock The classification of links up to link-homotopy.
\newblock {\em J. Am. Math. Soc.}, 3:389--419, 1990.

\bibitem{HabiroClasp}
K.~Habiro.
\newblock Claspers and finite type invariants of links.
\newblock {\em Geom. Topol.}, pages 1--83, 2000.

\bibitem{Mhall}
M.~Hall.
\newblock {\em The Theory of Groups}.
\newblock AMS Chelsea Publishing, 1959.

\bibitem{Phall}
P.~Hall.
\newblock {A contribution to the theory of groups of prime-power order}.
\newblock {\em {Proc. Lond. Math. Soc. (2)}}, 36:29--95, 1933.

\bibitem{Hughespartial}
J.~R. Hughes.
\newblock Partial conjugations suffice.
\newblock {\em Topology Appl.}, 148:55--62, 2005.

\bibitem{HumphriesTorsion}
S.~P. Humphries.
\newblock Torsion-free quotients of braid groups.
\newblock {\em Int. J. Algebra Comput.}, 11:363--373, 2001.

\bibitem{KirkLinkmaps4sphere}
P.~A. Kirk.
\newblock Link maps in the four sphere.
\newblock Differential topology, {Proc}. 2nd {Topology} {Symp}., {Siegen}/{FRG}
  1987, {Lect}. {Notes} {Math}. 1350, 31-43 (1988)., 1988.

\bibitem{KoschorkeLinkmapsclass}
U.~Koschorke.
\newblock On link maps and their homotopy classification.
\newblock {\em Math. Ann.}, 286:753--782, 1990.

\bibitem{kotorii_link-homotopy_2020}
Y.~Kotorii and A.~Mizusawa.
\newblock Link-homotopy classes of 4-component links and claspers.
\newblock {\em arXiv:1910.08653}, 2020.

\bibitem{Levine4comp}
J.~P. Levine.
\newblock An approach to homotopy classification of links.
\newblock {\em Trans. Am. Math. Soc.}, 306:361--387, 1988.

\bibitem{LyndonRogerSchuppPaulCobigrpthry}
R.~C. Lyndon and P.~E. Schupp.
\newblock {\em Combinatorial group theory.}
\newblock Berlin: Springer, 2001.

\bibitem{MagnusKarrassSolitar}
W.~Magnus, A.~Karrass, and D.~Solitar.
\newblock {\em Combinatorial group theory. {Presentations} of groups in terms
  of generators and relations.}
\newblock Mineola, NY: Dover Publications, 2004.

\bibitem{MasseyRolfsenHomoClassHigherdimLinks}
W.~S. Massey and D.~Rolfsen.
\newblock Homotopy classification of higher dimensional links.
\newblock {\em Indiana Univ. Math. J.}, 34:375--391, 1985.

\bibitem{JBYasuMilnorHOMPFLYT}
J-B. Meilhan and A.~Yasuhara.
\newblock Milnor invariants and the {HOMFLYPT} polynomial.
\newblock {\em Geom. Topol.}, 16:889--917, 2012.

\bibitem{MilnorLinkgrp}
J.~W. Milnor.
\newblock Link groups.
\newblock {\em Ann. Math. (2)}, 59:177--195, 1954.

\bibitem{MurasugiKunioKurpitaStudyofbraid}
K.~Murasugi and B.~I. Kurpita.
\newblock {\em A study of braids}, volume 484.
\newblock Dordrecht: Kluwer Academic Publishers, 1999.

\bibitem{OrrHomoInvLinks}
K.~E. Orr.
\newblock Homotopy invariants of links.
\newblock {\em Invent. Math.}, 95:379--394, 1989.

\bibitem{SchneidermanTeichnerGroupDisjoint2sphere4space}
R.~Schneiderman and P.~Teichner.
\newblock The group of disjoint 2-spheres in 4-space.
\newblock {\em Ann. Math. (2)}, 190:669--750, 2019.

\bibitem{YasuharaAkiraSelfdelteq}
A.~Yasuhara.
\newblock Self delta-equivalence for links whose {Milnor}'s isotopy invariants
  vanish.
\newblock {\em Trans. Am. Math. Soc.}, 361:4721--4749, 2009.

\end{thebibliography}
\end{document}